\documentclass{amsart}

\usepackage{amssymb}
\usepackage[all]{xy}
\usepackage{caption}
\usepackage[pdftex]{hyperref}
\usepackage{graphicx}
\usepackage{array}

\newtheorem{thm}{Theorem}
\newtheorem{lem}[thm]{Lemma}
\newtheorem{prop}[thm]{Proposition}
\newtheorem{cor}[thm]{Corollary}
\newtheorem{rem}[thm]{Remark}
\theoremstyle{definition}

\newcommand\scalemath[2]{\scalebox{#1}{\mbox{\ensuremath{\displaystyle #2}}}}

\newcommand{\pgl}{\PGL(2,\mathbb{C})}
\newcommand{\zz}{ \mathbb{Z}_{2}\times \mathbb{Z}_{2}}
\newcommand{\sldos}{\SL(2,\mathbb{C})}
\DeclareMathOperator{\tr}{tr}
\DeclareMathOperator{\Id}{Id}
\DeclareMathOperator{\Stab}{Stab}

\DeclareMathOperator{\GL}{GL}

\DeclareMathOperator{\SL}{SL}

\DeclareMathOperator{\PGL}{PGL}

\DeclareMathOperator{\Gr}{Gr}

\newcommand{\cM}{\mathcal{M}} %Moduli space% %moduli of parabolic bundles% %moduli of U(p,q) bundles%
\newcommand{\cH}{\mathcal{H}}
\newcommand{\x}{\times}
\newcommand{\CC}{\mathbb{C}} %Complex numbers%
\newcommand{\RR}{\mathbb{R}} %Real numbers%
\newcommand{\ZZ}{\mathbb{Z}} %Integer numbers%
\newcommand{\too}{\longrightarrow}
\newcommand{\imat}{\sqrt{-1}} %i%

\newcommand{\im}{\mathrm{im}\,}
\newcommand{\isom}{\cong}
\newcommand{\oX}{\overline{X}{}}
\newcommand{\oW}{\overline{W}{}}

\author[J. Mart\'{\i}nez]{Javier Mart\'{\i}nez}
\address{Facultad de Matem\'aticas, Universidad Complutense de Madrid,
Plaza Ciencias 3, 28040 Madrid, Spain}
\email{javiermartinez@mat.ucm.es}

\author[V. Mu\~{n}oz]{Vicente Mu\~{n}oz}
\address{Facultad de Matem\'aticas, Universidad Complutense de Madrid,
Plaza Ciencias 3, 28040 Madrid, Spain}
\address{Instituto de Ciencias Matem\'aticas (CSIC-UAM-UC3M-UCM),
C/ Nicolas Cabrera 15, 28049 Madrid, Spain}
\email{vicente.munoz@mat.ucm.es}

\title[E-polynomials of $\SL(2,\CC)$-character varieties]{E-polynomials 
of the $\SL(2,\CC)$-character varieties of surface groups}

\date{20 July 2014}



\begin{document}

\begin{abstract}
We compute the E-polynomials of the moduli spaces of representations of the
fundamental group of a once-punctured surface of any genus into $\SL(2,\CC)$,
for any possible holonomy around the puncture.
We follow the geometric technique introduced in \cite{lomune}, based on stratifying the space of
representations, and on the analysis of the behaviour of the E-polynomial under fibrations.
\end{abstract}

\maketitle

%%%%%%%%%%%%%%%%%%%%%%%%%%%%%%%%%%%%%%%%%%%%%%%%%%%%%%%%%%%%%%%%%%%%%%%%%%%%%%%%%%%%%%%
\section{Introduction}\label{sec:introduction}
%%%%%%%%%%%%%%%%%%%%%%%%%%%%%%%%%%%%%%%%%%%%%%%%%%%%%%%%%%%%%%%%%%%%%%%%%%%%%%%%%%%%%%%

Let $X$ be a smooth complex projective curve of genus $g\geq 1$, and let $G$ be a complex reductive group.
The $G$-character variety of $X$ is defined as the moduli space of semisimple representations of
$\pi_{1}(X)$ into $G$, that is,
 $$
  \cM (G)= \{(A_{1},B_{1},\ldots,A_{g},B_{g}) \in G^{2g} \,| \prod_{i=1}^{g}[A_{i},B_{i}]=\Id\}/ / G.
 $$
For the complex linear groups $G=\GL(n,\CC), \SL(n,\CC)$, the representations of $\pi_1(X)$ into $G$ can be understood
as $G$-local systems $E\to X$, hence defining
a flat bundle $E$ which has $\deg E=0$. A natural generalization consists of allowing bundles $E$ of
non-zero degree $d$. The $G$-local systems on $X$ correspond to representations $\rho:\pi_1(X - \{p_0\}) \to G$,
where $p_0\in X$ is a point, and $\rho(\gamma)= \frac{d}{n} \Id$, $\gamma$ a loop around $p_0$, giving
rise to the moduli space of twisted representations
 $$
  \cM^d (G)= \{(A_{1},B_{1},\ldots,A_{g},B_{g}) \in G^{2g} \,| \prod_{i=1}^{g}[A_{i},B_{i}]=e^{2\pi \imat d/n} \Id \}/ / G.
 $$
A related object is the moduli space of parabolic representations with one puncture, given by those 
representations whose monodromy around $p_0$ is $\rho(\gamma)=\xi=\text{diag}(\lambda_1,\ldots, \lambda_n)$,
where $\lambda_1\cdots \lambda_n =1$.
The parabolic character variety is
 $$
 \cM_{\xi}(G) = \{ (A_{1},B_{1},\ldots, A_{g},B_{g})\in G^{2g} \  |  
 \prod_{i=1}^{g}[A_{i},B_{i}] = \xi \} / /\Stab(\xi) .
 $$
These spaces are naturally generalized by the character varieties with arbitrary monodromy $C\in G$,
 $$
 \cM_{C}(G) = \{ (A_{1},B_{1},\ldots, A_{g},B_{g})\in G^{2g} \  |  
 \prod_{i=1}^{g}[A_{i},B_{i}] = C \} / /\Stab(C) .
 $$
Here $C$ can be central, diagonalizable, or of Jordan type.

\medskip

The space $\cM^d(G)$ is homeomorphic to the moduli space of  $G$-Higgs bundles  $\cH^d(G)$,
which parametrizes pairs $(E,\Phi)$, consisting of a vector bundle
$E\to X$ of degree $d$ and rank $n$ (with fixed determinant in the case $G=\SL(n,\CC)$),
and a homomorphism $\Phi:E\to E\otimes K_X$, known as the Higgs field (in the case $G=\SL(n,\CC)$, 
the Higgs field has trace $0$). This homeomorphism has been used to compute the 
cohomology of the moduli space $\cM^d(G)$ when $\gcd(n,d)=1$. 
Poincar\'e polynomials for $G=\SL(2,\CC)$ were computed in \cite{hitchin:1987}, for $G=\SL(3,\CC)$ 
in \cite{gothen:1994} and for $G=\GL(4,\CC)$ in \cite{garciaprada-heinloth-schmitt:2011}. A recursive formula for the motive of the moduli space of Higgs bundles of arbitrary rank and degree coprime to the rank has been given in \cite{garciaprada-heinloth:2013}. In particular, this gives the Betti numbers of the character variety for arbitrary coprime rank and degree.
There is a similar correspondence between the parabolic character variety and the moduli
space of parabolic Higgs bundles. The Poincar\'e polynomials of the moduli of parabolic 
Higgs bundles are given in \cite{by} for $G=\SL(2,\CC)$ and in \cite{garciaprada-gothen-munoz:2004} for
$G=\SL(3,\CC)$ and $\GL(3,\CC)$.

However, the homeomorphism $\cM^d(G) \isom \cH^d(G)$ is far from being an algebraic isomorphism, 
and hence the algebro-geometric information carried by these moduli spaces is different. Natural
refinements of the Poincar\'e polynomials which take into account the algebraic structure
are the E-polynomials (or Hodge-Deligne polynomials), which
encode the information of the dimensions of the Hodge decomposition of the cohomology in a nice way.
The fact that the E-polynomials of $\cM^d(G)$ and $\cH^d(G)$ have different natures has deep
implications  on the Mirror Symmetry phenomena exhibited in the non-abelian Hodge theory of a 
curve \cite{hausel-thaddeus:2003}.

Hausel and Rodriguez-Villegas started the computation of the E-polynomials of $G$-character
varieties focusing on $G=\GL(n,\CC), \SL(n,\CC)$ and $\PGL(n,\CC)$, using arithmetic methods inspired by the Weil
conjectures. In \cite{hausel-rvillegas:2007} they obtained the E-polynomials of $\cM^{d}(G)$
for $G=\GL(n,\CC)$ (i.e., for $C=e^{2\pi\imat d/n} \Id$), in terms of a simple generating function.
Following these methods, Mereb \cite{mereb:2010} studied this case for $\SL(n,\CC)$, giving
an explicit formula for the E-polynomial in the case $G=\SL(2,\CC)$, while for $\SL(n,\CC)$ 
these polynomials are given in terms of a generating function.
Also E-polynomials of the parabolic character varieties for $G=\GL(n,\CC)$ 
have been obtained by Hausel, Letellier and Rodriguez-Villegas, for semisimple conjugacy classes at the punctures
\cite{hausel-letellier-rvillegas:2011}.

Logares, Newstead and the second author introduced a geometric technique in \cite{lomune} 
to compute the E-polynomial of character varieties 
by using stratifications and also handling fibrations
which are locally trivial in the analytic topology but not in the Zariski topology. The
main results of \cite{lomune} are explicit formulas for the E-polynomials
of character varieties for $G=\SL(2,\CC)$ and $g=1,2$. Actually, the geometric technique
allows for dealing with character varieties in which the holonomy around the puncture
is not diagonalizable (in this case there is no correspondence with 
a Higgs bundle moduli space as mentioned above). In \cite{mamu}, the authors extend the
theory to compute the E-polynomial of the character variety for $G=\SL(2,\CC)$ and $g=3$,
and holonomy $\pm \Id$. Here we complete the general case $g\geq 3$ and any holonomy. 

We say that a variety $Z$ is of \textit{balanced type} (also called of Hodge-Tate type in the literature) if its mixed Hodge structure $H^{k,p,q}_c(Z)$ is non-zero only for $p=q$. Our first result is the following

\begin{thm}\label{thm:balanced}
All character varieties $\mathcal{M}_{C}(\sldos)$ are of balanced type.
\end{thm}

The main result of the paper is the computation of the E-polynomials of the character varietes for $\sldos$ and any genus.
\begin{thm}\label{thm:main}
  Let $X$ be a complex curve of genus $g\geq 1$. Let $\cM_C=\cM_C(\SL(2,\CC))$ be the
 character variety corresponding to $C\in \SL(2,\CC)$. The E-polynomials are as follows:
\begin{align*}
e(\mathcal{M}_{\Id})=  & \, (q^{3}-q)^{2g-2}+(q^{2}-1)^{2g-2}-q(q^{2}-q)^{2g-2}-2^{2g}q^{2g-2} \\
& +\frac{1}{2}q^{2g-2}(q+2^{2g}-1)((q+1)^{2g-2}+(q-1)^{2g-2}) + \frac{1}{2}q((q+1)^{2g-1}+(q-1)^{2g-1}) \\
\end{align*}
\begin{align*}
e(\mathcal{M}_{-\Id}) = & \,  (q^3-q)^{2g-2}+(q^2-1)^{2g-2}-2^{2g-1}(q^{2}+q)^{2g-2}+(2^{2g-1}-1)(q^{2}-q)^{2g-2} \\
e(\mathcal{M}_{J_{+}}) = & \, (q^{3}-q)^{2g-2}(q^{2}-1)+(2^{2g-1}-1)(q-1)(q^{2}-q)^{2g-2}-2^{2g-1}(q+1)(q^{2}+q)^{2g-2}\\
& +\frac{1}{2}q^{2g-2}(q-1)\left((q-1)^{2g-1}-(q+1)^{2g-1} \right) \\
e(\mathcal{M}_{J_{-}}) =& \, (q^{3}-q)^{2g-2}(q^{2}-1)+(2^{2g-1}-1)(q-1)(q^{2}-q)^{2g-2}+2^{2g-1}(q+1)(q^{2}+q)^{2g-2} \\
e(\mathcal{M}_{\xi_{\lambda}}) = & \,  (q^{3}-q)^{2g-2}(q^{2}+q)+(q^{2}-1)^{2g-2}(q+1)+(2^{2g}-2)(q^{2}-q)^{2g-2}q,
\end{align*}
for $J_+=\begin{pmatrix} 1& 1\\ 0 &1\end{pmatrix}$,
$J_-=\begin{pmatrix} - 1& 1\\ 0 &- 1\end{pmatrix}$ and
$\xi_\lambda=\begin{pmatrix} \lambda & 0\\ 0 &\lambda^{-1}\end{pmatrix}$, $\lambda\neq 0,\pm 1$,
and with $q=\, uv$. 
\end{thm}

This generalizes the formulas of  \cite{lomune} for $g=1,2$ (the formula for $e(\cM_{\Id})$ for $g=2$ in Theorem 1.2 of
\cite{lomune} has a misprint; the polynomial should be $e(\cM_{\Id})=q^6+ 17q^4+q^2+ 1$, 
which is correctly written in Section 8 of \cite{lomune}).  It also generalizes the formulas in \cite{mamu} for 
$e(\cM_{\Id})$ and $e(\cM_{-\Id})$ for $g=3$. The formula for $e(\cM_{-\Id})$ and any $g$ coincides with
that of \cite{mereb:2010}. The other E-polynomials are new, including the case of parabolic $\sldos$-character varieties $\mathcal{M}_{\xi_\lambda}$.

We use the information for the cases $g=1,2$ from \cite{lomune} 
as building blocks. The E-polynomial for the character variety for a curve $X_g$  of
genus $g\geq 3$ will be computed inductively. The basic idea, as will be clear throughout the paper,
is to decompose $X_g=X_{g-1}\# X_1$ as a connected sum. From the information for
$X_g$ one gets information for $X_{g-1}$ with a hole, and this is used in turn to 
compute the E-polynomial corresponding to $X_{g+1}=X_{g-1}\# X_2$. The
E-polynomials of $X_1,X_2$ with a puncture from \cite{lomune} come into play here.

This induction has as starting point the curve $X_3$ of genus $g=3$ (with no
puncture), which is computed in \cite{mamu}. This special case has its special
features, and has to be treated separately. In particular, in that case the 
techniques to compute E-polynomials of analytically locally trivial fibrations use
a base of dimension $2$. In all other cases ($g=1,2$ in \cite{lomune}, and 
the induction for $g\geq 4$ treated here) a base of dimension $1$ suffices.

Theorem \ref{thm:main} allows us to prove the following relation 
of the E-polynomials of various character varieties, conjectured by T. Hausel.

\begin{cor} \label{cor:Hausel}
For any genus $g\geq 1$, we have
$$
e(\mathcal{M}_{J_{-}})+(q+1)e(\mathcal{M}_{-\Id})=e(\mathcal{M}_{\xi_{\lambda}}).
$$
\end{cor}

As a byproduct of our analysis, we also obtain the behaviour of the E-polynomial of the parabolic character
variety ($G=\SL(2,\CC)$)
 $$
  \cM_{\xi_\lambda}=\cM_{\xi_\lambda} (G)= \{(A_{1},B_{1},\ldots,A_{g},B_{g}) \in G^{2g} \,| \prod_{i=1}^{g}[A_{i},B_{i}]=
  \begin{pmatrix} \lambda & 0 \\ 0 &\lambda^{-1} \end{pmatrix} \}/ / G,
 $$
when $\lambda$ varies in $\CC-\{0,\pm 1\}$. This is of relevance for Mirror Symmetry phenomena. It is given by the following formula.

\begin{thm}
Let $X$ be a curve of genus $g\geq 1$. Then
  \begin{equation*}\label{eqn:RcM}
R(\mathcal{M}_{\xi_{\lambda}})= \left( (q^{3}-q)^{2g-2}(q^{2}+q)+(q+1)(q^{2}-1)^{2g-2}-q(q^{2}-q)^{2g-2} \right) T 
+ \left( (2^{2g}-1)q(q^{2}-q)^{2g-2} \right) N,
  \end{equation*}
which means that the E-polynomial of the invariant part of the cohomology is the polynomial accompanying $T$,
and the  E-polynomial of the non-invariant part is the polynomial accompanying $N$.  
\end{thm}

We end up giving some consequences of Theorem \ref{thm:main} in Section \ref{sec:topol}. Notably, 
\begin{cor}
The E-polynomials of $e(\mathcal{M}_{-\Id})$, $e(\mathcal{M}_{\xi_{\lambda}})$ are palindromic.
\end{cor}

\medskip

The present arguments can be used to compute the E-polynomials of the $\PGL(2,\CC)$-character varieties of 
surface groups for arbitrary genus, which will appear in \cite{ma}.

\noindent\textbf{Acknowledgements.} 
We would like to thank Peter Newstead, Tam\'as Hausel, Martin Mereb, Nigel Hitchin, Marina Logares and Oscar Garc\'{\i}a-Prada for useful conversations.
This work has been partially supported by MICINN (Spain) Project MTM2010-17389. The first author was also supported by a FPU scholarship from the Spanish Ministerio de Educaci\'on.

%%%%%%%%%%%%%%%%%%%%%%%%%%%%%%%%%%%%%%%%%%%%%%%%%%%%%%%%%%%%%%%%%%%%%%%%% 
\section{E-polynomials}\label{sec:method}
%%%%%%%%%%%%%%%%%%%%%%%%%%%%%%%%%%%%%%%%%%%%%%%%%%%%%%%%%%%%%%%%%%%%%%%%% 

We start by  giving the definition of E-polynomials. 
A pure Hodge structure of weight $k$ consists of a finite dimensional complex vector space
$H$ with a real structure, and a decomposition $H=\bigoplus_{k=p+q} H^{p,q}$
such that $H^{q,p}=\overline{H^{p,q}}$, the bar meaning complex conjugation on $H$.
A Hodge structure of weight $k$ gives rise to the so-called Hodge filtration, which is a descending filtration
$F^{p}=\bigoplus_{s\ge p}H^{s,k-s}$. We define $\Gr^{p}_{F}(H):=F^{p}/ F^{p+1}=H^{p,k-p}$.
A mixed Hodge structure consists of a finite dimensional complex vector space $H$ with a real structure,
an ascending (weight) filtration $\ldots \subset W_{k-1}\subset W_k \subset \ldots \subset H$
(defined over $\RR$) and a descending (Hodge) filtration $F$ such that $F$ induces a pure Hodge structure
of weight $k$ on each $\Gr^{W}_{k}(H)=W_{k}/W_{k-1}$. We define
 $$
 H^{p,q}:= \Gr^{p}_{F}\Gr^{W}_{p+q}(H)
 $$
and write $h^{p,q}$ for the {\em Hodge number} $h^{p,q} :=\dim H^{p,q}$.

Let $Z$ be any quasi-projective algebraic variety (maybe non-smooth or non-compact). 
The cohomology groups $H^k(Z)$ and the cohomology groups with compact support  
$H^k_c(Z)$ are endowed with mixed Hodge structures \cite{Deligne2,Deligne3}. 
We define the {\em Hodge numbers} of $Z$ by
$h^{k,p,q}_{c}(Z) = h^{p,q}(H_{c}^k(Z))$.
The Hodge-Deligne polynomial, or E-polynomial, is defined as 
 $$
 e(Z)=e(Z)(u,v):=\sum _{p,q,k} (-1)^{k}h^{k,p,q}_{c}(Z) u^{p}v^{q}.
 $$
When $h_c^{k,p,q}=0$ for $p\neq q$, the polynomial $e(Z)$ depends only on the product $uv$.
This will happen in all the cases that we shall investigate here. In this situation,
we use the variable $q=uv$. If this happens, we say that the variety is {\it of balanced type}.
For instance, $e(\CC^n)=q^n$.

The key property of Hodge-Deligne polynomials that permits their calculation is that they are additive for
stratifications of $Z$. If $Z$ is a complex algebraic variety and
$Z=\bigsqcup_{i=1}^{n}Z_{i}$, where all $Z_i$ are locally closed in $Z$, then $e(Z)=\sum_{i=1}^{n}e(Z_{i})$.

There is another useful property that we shall use often: if there is an action of $\ZZ_2$ on $X$, then we have
polynomials $e(X)^+, e(X)^-$, which are the E-polynomials of the invariant and anti-invariant parts of the cohomology
of $X$, respectively. More concretely, $e(X)^+=e(X/\ZZ_2)$ and $e(X)^-=e(X)-e(X)^+$. Then if $\ZZ_2$ acts
on $X$ and on $Y$, we have the equality (see \cite[Proposition 2.6]{lomune}) 
 \begin{equation}\label{eqn:+-}
e((X\x Y)/\ZZ_2) =e(X)^+e(Y)^+ +e(X)^-e(Y)^-\, .
  \end{equation}

Suppose that
 \begin{equation}\label{fibration}
  F \longrightarrow Z \overset{\pi}{\longrightarrow} B 
 \end{equation}
is a fibration locally trivial in the analytic topology, and with $F$ of balanced type.
The fibration defines a local system $\mathcal{H}^{k}_{c}$, 
whose fibers are the cohomology groups $H^{k}_{c}(F_{b})$, where $b\in B$, $F_{b}=\pi^{-1}(b)$. 
Associated to the fibration, there is a monodromy representation
 \begin{equation}\label{eqn:exxtra}
 \rho : \pi_{1}(B) \longrightarrow  \GL(H^{k,p,p}_c(F)).
 \end{equation}
Suppose that the  monodromy group $\Gamma=\im(\rho)$ is an abelian and finite group. 
Then $H_c^{k,p,p}(F)$ are modules over the representation ring $R(\Gamma)$. So there is a well
defined element, the  {\em Hodge monodromy representation},
  \begin{equation}\label{eqn:Hodge-mon-rep}
  R(Z) := \sum (-1)^k H_c^{k,p,p}(F)\,  q^p \in R(\Gamma)[q] \, .
  \end{equation}
As the monodromy representation (\ref{eqn:exxtra}) has finite image, there is a finite covering
$B_\rho \to B$ such that the pull-back fibration has trivial monodromy. We have the
following result.

\begin{thm}[{\cite[Theorem 2]{mamu}}] \label{thm:general-fibr}
 Suppose that $B_\rho$ is of balanced type. Then $Z$ is of balanced type. 
There is a $\ZZ[q]$-linear map
 $$
 e: R(\Gamma)[q] \to \ZZ[q]
 $$
satisfying the property that $e(R(Z))=e(Z)$
\end{thm}

This can be rephrased as follows.
Let $S_1,\ldots, S_N$ be the
irreducible representations of $\Gamma$ (there are $N=\# \Gamma$ of them, and all
of them are one-dimensional). These are generators of $R(\Gamma)$ as a free abelian group. 
Write $s_i(q)=e(S_i)$, $1\leq i\leq N$. Then if we
write the Hodge monodromy representation of (\ref{eqn:Hodge-mon-rep}) as
  $$
 R(Z)= a_1(q) S_1 + \ldots a_N(q) S_N,
 $$
then we have 
 $$
 e(Z)= a_1(q) s_1(q) + \ldots +a_N(q) s_N(q) .
 $$

We shall only use 
Theorem \ref{thm:general-fibr} in two situations. First, when 
$\pi:Z\to B$ is a fibre bundle with fibre $F$ such that the action of $\pi_1(B)$ on $H_c^*(F)$ is trivial.
Then $R(Z)=e(F) T$, where $T$ is the trivial local system and
$e(Z)=e(F) e(B)$ (this result appears in \cite[Proposition 2.4]{lomune}).
In particular, this happens when 
$Z$ is a $G$-space with isotropy $H<G$ such that $G/H\to Z\to B$ is a fiber bundle, 
and $G$, $H$ are connected algebraic groups. Then $e(Z)=e(B)e(G)/e(H)$

Second, when $B$ is one-dimensional. So we have a fibration 
 $$
 F\too Z \too B=\CC-\{q_1,\ldots, q_\ell\}, 
 $$
with monodromy $\rho$ and Hodge monodromy representation $e(R(Z))\in R(\Gamma)[q]$.
Assuming that $B_\rho$ is a rational curve, we write
$R(Z)=a_1(q) T+ a_2(q) S_2+ \ldots + a_N(q)S_N$, where $T$ is the trivial representation
and $S_2, \ldots, S_N$ are the non-trivial representations. Then $e(T)=q-\ell$ and $e(S_i)=-(\ell-1)$,
$2 \leq i\leq N$. Hence
 \begin{equation}\label{eqn:dos}
 e(Z)=(q- \ell) a_1(q) -(\ell -1) \sum_{i=2}^N a_i(q)=(q-1)\, e(F)^{inv} - (\ell-1) e(F),
 \end{equation}
where $e(F)^{inv}=a_1(q)$ is the E-polynomial of
the invariant part of the cohomology of $F$ and $e(F)=\sum_{i=1}^N a_i(q)$. See \cite[Corollary 3]{mamu}.

%%%%%%%%%%%%%%%%%%%%%%%%%%%%%%%%%%%%%%%%%%%%%
\section{Stratifying the space of representations} \label{sec:strata}
%%%%%%%%%%%%%%%%%%%%%%%%%%%%%%%%%%%%%%%%%%%%%

Let $g\geq 1$ be any natural number. We define the following sets:
\begin{itemize}
\item $\oX_{0}^{g}= \lbrace (A_{1},B_{1},\ldots, A_{g},B_{g}) \in \SL(2,\CC)^{2g} \mid \prod_{i=1}^{g}[A_{i},B_{i}]= \Id\}$.
\item $\oX_{1}^{g}= \lbrace (A_{1},B_{1},\ldots, A_{g},B_{g})\in \SL(2,\CC)^{2g} \mid \prod_{i=1}^{g}[A_{i},B_{i}]= -\Id\}$.
\item $\oX_{2}^{g}= \lbrace (A_{1},B_{1},\ldots, A_{g},B_{g}) \in \SL(2,\CC)^{2g}\mid \prod_{i=1}^{g}[A_{i},B_{i}]= J_+=
 \begin{pmatrix} 1& 1\\ 0 & 1 \end{pmatrix}\}$.
\item $\oX_{3}^{g}= \lbrace (A_{1},B_{1},\ldots, A_{g},B_{g}) \in \SL(2,\CC)^{2g}\mid \prod_{i=1}^{g}[A_{i},B_{i}]=  J_-=
 \begin{pmatrix} -1& 1\\ 0 & -1 \end{pmatrix}\}$.
\item $\oX_{4,\lambda}^{g}= 
\lbrace (A_{1},B_{1},\ldots, A_{g},B_{g}) \in \SL(2,\CC)^{2g}\mid \prod_{i=1}^{g}[A_{i},B_{i}]= \xi_\lambda=
\begin{pmatrix} \lambda & 0 \\ 0 & \lambda^{-1} \end{pmatrix} \rbrace$, where $\lambda\in \CC-\{ 0, \pm 1\}$.
\item $\oX_{4}^{g}= 
\lbrace (A_{1},B_{1},\ldots, A_{g},B_{g},\lambda) \in \SL(2,\CC)^{2g}\x (\CC-\{ 0, \pm 1\}) \mid 
\prod_{i=1}^{g}[A_{i},B_{i}]= \begin{pmatrix} \lambda & 0 \\ 0 & \lambda^{-1} \end{pmatrix} \rbrace$.
\end{itemize}

There is a natural fibration 
 $$ 
  \oX_4^g  \too  \CC-\{ 0, \pm 1\}
 $$
whose fiber are $\oX_{4,\lambda}^{g}$.
There is an action of $\ZZ_2$ on $\oX_4$ given by
$ (A_{1},\ldots,B_{g},\lambda) \mapsto (P^{-1}_0A_1P_0,\ldots, P^{-1}_0B_gP_0,\lambda^{-1})$, with $P_0=
\left(\begin{array}{cc}    0 & 1 \\ 1 & 0  \end{array} \right)$. There is an induced fibration
 $$ 
 \oX_4^g/\ZZ_2  \too  ( \CC-\{ 0, \pm 1\})/\ZZ_2 \cong \CC-\{\pm 2\},
 $$
where the basis is parametrized by $t=\lambda+\lambda^{-1}$.

For $g=1$, we shall denote $\oX_0=\oX_0^1, \oX_1=\oX_1^1$, etc., following the notations of \cite{lomune} and \cite{mamu}.
For $g=2$, we shall denote   $\overline{Y}_0=\oX_0^2, \overline{Y}_1=\oX_1^2$, etc., following the notations of \cite[Section 4]{mamu}.
We collect the information for $g=1,2$ from \cite{lomune} and \cite{mamu}.
\begin{itemize}
\item $e(\oX_{0}) =  q^{4}+4q^{3}-q^2-4q$
\item  $e(\oX_1)  =q^{3}-q $
\item  $e(\oX_2) = q^3-2q^2-3 q $
\item  $e(\oX_3) =q^3 +3 q^2 $
\item  $R(\oX_{4}/\mathbb{Z}_{2})=q^{3}T-3q S_{2}+3q^{2}S_{-2}-S_{0}$
\item $e(\overline{Y}_0)= q^9+q^8+12q^7+2q^6-3q^4-12q^3-q$.
\item  $e(\overline{Y}_1)= q^9-3q^7-30q^6+30q^4+3q^3-q$.
\item $e(\overline{Y}_2)= q^9-3q^7-4q^6-39q^5-4q^4-15q^3$.
\item $e(\overline{Y}_3)= q^9-3q^7+15q^6+6q^5+45q^4$.
\item $R(\overline{Y}_{4}/\ZZ_2)=(q^9-3q^7+6q^5)T-(45q^5+15q^3)S_{2}+(15q^6+45q^4)S_{-2}+(-6q^4+3q^2-1)S_{0}$.
\end{itemize}
Here the monodromy group is $\Gamma=\ZZ_2\x\ZZ_2$, generated by the loops $\gamma_{\pm 2}$ around the punctures $\pm 2$ of 
$\CC-\{\pm 2\}$. The ring $R(\Gamma)$ is generated by the trivial representation $T$, the representations $S_{\pm 2}$ which 
are non-trivial around $\pm2$ and trivial around $\mp 2$, and the representation $S_0=S_2\otimes S_{-2}$.

\medskip

Now we shall set up an induction. Assume that for all $k< g$, the Hodge monodromy representation of $\oX_4^k/\ZZ_2$
is in $R(\Gamma)[q]$. We write
 $$
 R(\oX_4^k/\ZZ_2)=a_k T+ b_k S_2 + c_k S_{-2} + d_k S_0,
$$
for some polynomials $a_k,b_k,c_k,d_k \in  \ZZ[q]$.
 
Take $k,h< g$. Fix some $C=\Id, -\Id, J_+,J_-$ or $\xi_\lambda$. Then 
 \begin{equation} \label{eqn:uno}
\prod_{i=1}^{k+h}[A_{i},B_{i}]= C  \iff \prod_{i=1}^{k}[A_{i},B_{i}]= C \prod_{i=1}^{h}[B_{k+i},A_{k+i}].
 \end{equation}

%%%%%%%%%%%%%%%%%%%%%%%%%%%%%%%%%%%%%%%%%%%%%%%%%%%%
\section{Computation of $e(\oX_0^{k+h})$}\label{subsec:0}
%%%%%%%%%%%%%%%%%%%%%%%%%%%%%%%%%%%%%%%%%%%%%%%%%%%%

Using (\ref{eqn:uno}), we stratify $\oX_{0}^{k+h}= \bigsqcup W_{i}$, where
\begin{itemize}
\item $W_{0}=\lbrace (A_{1},B_1,\ldots,A_{k+h},B_{k+h}) \mid \prod_{i=1}^{k}[A_{i},B_{i}]=  \prod_{i=1}^{h}[B_{k+i},A_{k+i}] =\Id \rbrace \cong \oX_{0}^{k}\times  \oX_{0}^h$.
\item $W_{1}=\lbrace (A_{1},B_1,\ldots,A_{k+h},B_{k+h}) \mid \prod_{i=1}^{k}[A_{i},B_{i}]=  \prod_{i=1}^{h}[B_{k+i},A_{k+i}] =-\Id \rbrace \cong \oX_{1}^{k}\times  \oX_{1}^h$.
\item $W_2=\lbrace (A_{1},B_1,\ldots,A_{k+h},B_{k+h}) \mid \prod_{i=1}^{k}[A_{i},B_{i}]=  \prod_{i=1}^{h}[B_{k+i},A_{k+i}] \sim J_{+} \rbrace \cong
 \PGL(2,\CC)/U \times \oX_{2}^{k}\times \oX_{2}^h$, where $U=\{ \begin{pmatrix} 1 & x \\ 0 & 1 \end{pmatrix}\} \cong  \CC$.
\item $W_3=\lbrace (A_{1},B_1,\ldots,A_{k+h},B_{k+h}) \mid \prod_{i=1}^{k}[A_{i},B_{i}]=  \prod_{i=1}^{h}[B_{k+i},A_{k+i}] \sim J_{-} \rbrace\cong
 \PGL(2,\CC)/U \times \oX_{3}^{k}\times \oX_{3}^h$.
\item $W_{4}= \lbrace (A_{1},B_1,\ldots,A_{k+h},B_{k+h}) \mid \prod_{i=1}^{k}[A_{i},B_{i}]=  \prod_{i=1}^{h}[B_{k+i},A_{k+i}] \sim
 \begin{pmatrix} \lambda & 0 \\ 0 & \lambda^{-1} \end{pmatrix}, \lambda \neq 0, \pm 1 \rbrace$.
\end{itemize}

To compute $e(W_4)$, we define  
 $$
 \oW_4= \lbrace (A_{1},B_1,\ldots,A_{k+h},B_{k+h},\lambda) \mid \prod_{i=1}^{k}[A_{i},B_{i}]=  \prod_{i=1}^{h}[B_{k+i},A_{k+i}]=
 \begin{pmatrix} \lambda & 0 \\ 0 & \lambda^{-1} \end{pmatrix}, \lambda \neq 0, \pm 1 \rbrace,
 $$ 
which produces the fibration 
  $$
 \oW_{4}/\ZZ_2 \rightarrow \CC -\{ \pm 2\},
 $$ 
whose Hodge monodromy representation is
\begin{align} \label{eqn:tres}
R(\oW_{4}/\ZZ_2) =& \, R(\oX_{4}^{k}/\ZZ_2) \otimes R(\oX_{4}^h/\ZZ_2) \nonumber \\
 = &\,  (a_{k}a_h+b_{k}b_h+c_{k}c_h+d_{k}d_h)T + (a_{k}b_h+b_{k}a_h+c_{k}d_h+d_{k}c_h)S_{2} \\
& + (a_{k}c_h+b_{k}d_h+c_{k}a_h+d_{k}b_h)S_{-2} + (a_{k}d_h+b_{k}c_h+c_{k}b_h+d_{k}a_h)S_{0}. \nonumber
\end{align}

From $R(\oW_{4}/\ZZ_2)$ there is a standard procedure to obtain $e(W_4)$. 
For brevity, write $R(\oW_{4}/\ZZ_2)=A T+ BS_2+ CS_{-2}+DS_0$, where
\begin{align} \label{eqn:tres.5}
 A&=  a_{k}a_h+b_{k}b_h+c_{k}c_h+d_{k}d_h\nonumber \\
 B&= a_{k}b_h+b_{k}a_h+c_{k}d_h+d_{k}c_h \\
 C &= a_{k}c_h+b_{k}d_h+c_{k}a_h+d_{k}b_h \nonumber\\
 D&= a_{k}d_h+b_{k}c_h+c_{k}b_h+d_{k}a_h \nonumber
\end{align}
First, using (\ref{eqn:dos}), we have that $e(\oW_{4}/\ZZ_2)=(q-2)A-(B+C+D)$.
On the other hand, the $2:1$-cover $\CC-\{0,\pm1\} \to \CC -\{ \pm 2\}$ allows us to
deduce that $R(\oW_{4})=(A+D) T+ (B+C)N$, where $T$ is the trivial representation,
and $N$ is the representation which is non-trivial and of order two around the origin. So
using  (\ref{eqn:dos}) again, we have that $e(\oW_{4})=(q-3)(A+D)-2(B+C)$.

Now note that
 $$
 W_4 \cong (\PGL(2,\CC) /D \x \oW_{4}) /\ZZ_2,
 $$
where $D\cong \CC^*$ are the diagonal matrices.
By \cite[Proposition 3.2]{lomune}, we have that $ e(\PGL(2, \CC)/D)^+=q^2$ and $ e(\PGL(2, \CC)/D)^-=q$. Using (\ref{eqn:+-}),
 \begin{align}\label{eqn:RX4-RX4Z2->eX4}
 e(W_4) &= q^2 e(\oW_ 4)^+ + q \, e(\oW_4)^- \nonumber \\
  &= q^2 e(\oW_4/\ZZ_2) + q( e(\oW_4)-e(\oW_4/\ZZ_2))  \nonumber \\
 &= (q^2-q)e(\oW_4/\ZZ_2) +q \, e(\oW_{4}) \\
 &= (q^2-q)((q-2)A-(B+C+D)) + q ((q-3)(A+D)-2(B+C))  \nonumber \\
 &= (q^3-2q^2-q)A-(q^2+q)(B+C)-2qD. \nonumber 
 \end{align}

All together, recalling also that $e(\PGL(2,\CC))=q^3-q$ and so $e(\PGL(2,\CC)/U)=q^2-1$, we have
 \begin{align} \label{eqn:eX0} 
 e(\oX_{0}^{k+h}) &=  e(\oX_0^k)e(\oX_0^h)+ e(\oX_1^k)e(\oX_1^h)+
 (q^2-1) e(\oX_2^k)e(\oX_2^h)+ (q^2-1) e(\oX_3^k)e(\oX_3^h)+ e(W_4).
 \end{align}

Setting $k=g-1$, $h=1$, and substituting the values $A,B,C,D$ from (\ref{eqn:tres.5}) into (\ref{eqn:RX4-RX4Z2->eX4}),
and then the values of $e(\oX_j^1)$ and $a_1,b_1,c_1,d_1$ from Section \ref{sec:strata}, we have
 \begin{align}\label{eqn:eX0.} \tag{$\alpha$}
 e_0^g=e(\oX_{0}^{g}) = & \, (q^4 + 4 q^3 - q^2 - 4 q) e_0^{g-1}+ (q^3-q) e_1^{g-1} \nonumber \\ 
&+ ( q^5- 2 q^4 - 4 q^3+ 2 q^2  +3 q ) e_2^{g-1}+ (q^5+3q^4-q^3-3q^2) e_3^{g-1}  \nonumber \\ 
&+ (q^6-2q^5-4q^4+3q^2+2q) a_{g-1} + (-q^5-4q^4+4q^2+q)b_{g-1}   \nonumber\\
&+ (2q^5-7q^4-3q^3+7q^2-q) c_{g-1} + (-5q^4-q^3+5q^2-q) d_{g-1} \nonumber
 \end{align}

%%%%%%%%%%%%%%%%%%%%%%%%%%%%%%%%%%%%%%%%%%%%%%%%%%%%
\section{Computation of $e(\oX_1^{k+h})$}\label{subsec:1}
%%%%%%%%%%%%%%%%%%%%%%%%%%%%%%%%%%%%%%%%%%%%%%%%%%%%

We do something similar to the previous case. We stratify $\oX_{1}^{k+h}= \bigsqcup W'_{i}$, where
\begin{itemize}
\item $W'_{0}=\lbrace (A_{1},B_1,\ldots,A_{k+h},B_{k+h}) \mid \prod_{i=1}^{k}[A_{i},B_{i}]= - \prod_{i=1}^{h}[B_{k+i},A_{k+i}] =\Id \rbrace \cong \oX_{0}^{k}\times  \oX_{1}^h$.
\item $W'_{1}=\lbrace (A_{1},B_1,\ldots,A_{k+h},B_{k+h}) \mid \prod_{i=1}^{k}[A_{i},B_{i}]=  -\prod_{i=1}^{h}[B_{k+i},A_{k+i}] =-\Id \rbrace \cong \oX_{1}^{k}\times  \oX_{0}^h$.
\item $W'_2=\lbrace (A_{1},B_1,\ldots,A_{k+h},B_{k+h}) \mid \prod_{i=1}^{k}[A_{i},B_{i}]=  -\prod_{i=1}^{h}[B_{k+i},A_{k+i}] \sim J_{+} \rbrace
 \cong\PGL(2,\CC)/U \times \oX_{2}^{k}\times \oX_{3}^h$, where $U=\{ \begin{pmatrix} 1 & x \\ 0 & 1 \end{pmatrix}\} \cong  \CC$.
\item $W'_3=\lbrace (A_{1},B_1,\ldots,A_{k+h},B_{k+h}) \mid \prod_{i=1}^{k}[A_{i},B_{i}]= - \prod_{i=1}^{h}[B_{k+i},A_{k+i}] \sim J_{-} \rbrace
 \cong\PGL(2,\CC)/U \times \oX_{3}^{k}\times \oX_{2}^h$.
\item $W'_{4}= \lbrace (A_{1},B_1,\ldots,A_{k+h},B_{k+h}) \mid \prod_{i=1}^{k}[A_{i},B_{i}]= - \prod_{i=1}^{h}[B_{k+i},A_{k+i}] \sim
 \begin{pmatrix} \lambda & 0 \\ 0 & \lambda^{-1} \end{pmatrix}, \lambda \neq 0, \pm 1 \rbrace$.
\end{itemize}

To compute $e(W'_4)$, we define  $\oW'_4=
\lbrace (A_{1},B_1,\ldots,A_{k+h},B_{k+h},\lambda) \mid \prod_{i=1}^{k}[A_{i},B_{i}]= - \prod_{i=1}^{h}[B_{k+i},A_{k+i}]=
 \begin{pmatrix} \lambda & 0 \\ 0 & \lambda^{-1} \end{pmatrix}, \lambda \neq 0, \pm 1 \rbrace$, which
produces the fibration 
  $$
 \oW'_{4}/\ZZ_2 \rightarrow \CC -\{ \pm 2\},
 $$ 
whose Hodge monodromy representation is given as (where $\tau(\lambda)=-\lambda$), 
\begin{align} \label{eqn:tres'}
R(\oW'_{4}/\ZZ_2) =& \, R(\oX_{4}^{k}/\ZZ_2) \otimes \tau^* R(\oX_{4}^h/\ZZ_2) \nonumber \\
 = & \, (a_{k}T+b_{k}S_{2}+c_{k}S_{-2}+d_{k}S_{0}) \otimes (a_hT+c_hS_{2}+b_hS_{-2}+d_hS_{0}) \nonumber\\
 = & \, (a_{k}a_h+b_{k}c_h+c_{k}b_h+d_{k}d_h)T + (a_{k}c_h+b_{k}a_h+c_{k}d_h+d_{k}b_h)S_{2}\\
& + (a_{k}b_h+b_{k}d_h+c_{k}a_h+d_{k}c_h)S_{-2} + (a_{k}d_h+b_{k}b_h+c_{k}c_h+d_{k}a_h)S_{0} \nonumber
\end{align}
We write $R(\oW'_{4}/\ZZ_2)=A' T+ B'S_2+ C'S_{-2}+D'S_0$, with
\begin{align} \label{eqn:tres'.5}
 A'&=  a_{k}a_h+b_{k}c_h+c_{k}b_h+d_{k}d_h\nonumber \\
 B'&= a_{k}c_h+b_{k}a_h+c_{k}d_h+d_{k}b_h \\
 C' &=a_{k}b_h+b_{k}d_h+c_{k}a_h+d_{k}c_h \nonumber\\
 D'&=a_{k}d_h+b_{k}b_h+c_{k}c_h+d_{k}a_h \nonumber
\end{align}
We use (\ref{eqn:RX4-RX4Z2->eX4}) to get
 \begin{align*}
  e(W'_4) &= (q^3-2q^2-q)A'-(q^2+q)(B'+C')-2qD'.
 \end{align*}
Finally
 \begin{align} \label{eqn:eX1} 
 e(\oX_{1}^{k+h}) &=  e(\oX_0^k)e(\oX_1^h)+ e(\oX_1^k)e(\oX_0^h)+
 (q^2-1) e(\oX_2^k)e(\oX_3^h)+ (q^2-1) e(\oX_3^k)e(\oX_2^h)+ e(W'_4).
 \end{align}

Setting $k=g-1$, $h=1$, and substituting the values $A',B',C',D'$ from (\ref{eqn:tres'.5})
and the values of $e(\oX_j^1)$ and $a_1,b_1,c_1,d_1$ from Section \ref{sec:strata}, we have
 \begin{align}\label{eqn:eX1.} \tag{$\beta$}
  e_1^g=e(\oX_{1}^{g}) = & \,  (q^3-q)e_0^{g-1}+ (q^4 + 4 q^3 - q^2 - 4 q) e_1^{g-1} \nonumber \\ 
&+  (q^5+3q^4-q^3-3q^2)e_2^{g-1}+ ( q^5- 2 q^4 - 4 q^3+ 2 q^2  +3 q ) e_3^{g-1}  \nonumber \\ 
&+(q^6-2q^5-4q^4+3q^2+2q) a_{g-1} + (2q^5-7q^4-3q^3+7q^2+q) b_{g-1} \nonumber\\
&+ (-q^5-4q^4+4q^2+q) c_{g-1} + (-5q^4-q^3+5q^2+q) d_{g-1} \nonumber
 \end{align}

%%%%%%%%%%%%%%%%%%%%%%%%%%%%%%%%%%%%%%%%%%%%%%%%%%%%
\section{Computation of $e(\oX_2^{k+h})$}\label{subsec:2}
%%%%%%%%%%%%%%%%%%%%%%%%%%%%%%%%%%%%%%%%%%%%%%%%%%%%

Now we consider
$$
Z=\lbrace (A_{1},B_{1},\ldots, A_{k+h},B_{k+h}) \mid \prod_{i=1}^{k}[A_{i},B_{i}]= J_{+}\prod_{i=1}^{h}[B_{k+i},A_{k+i}] \rbrace.
$$
We write
$$
\nu =\prod_{i=1}^{h}[B_{k+i},A_{k+i}] =\begin{pmatrix} a & b \\ c & d \end{pmatrix}, \qquad 
\delta = \prod_{i=1}^{k}[A_{i},B_{i}] = J_{+}\nu = \begin{pmatrix} a+c & b+d \\ c & d \end{pmatrix}.
$$
Let $t_{1}=\tr \nu$, $t_{2}=\tr \delta$. Note that $c=t_{2}-t_{1}$. 

We use the stratification defined as in  \cite[Section 11]{lomune}, according to the values of the pair $(t_{1},t_{2})$. 
\begin{itemize}
\item $Z_{1}$ given by $(t_{1},t_{2})=(2,2)$. In this case $c=0$, $a=d=1, b\in \CC$. If $b\neq 0,-1$, both are of Jordan type, 
whereas if $b=0,1$ one of them is of Jordan type and the other is equal to $\Id$. We get
$$
e(Z_{1})=(q-2)e(\oX^k_{2})e(\oX_{2}^{h})+e(\oX_{2}^k)e(\oX_{0}^{h})+e(\oX_{0}^k)e(\oX_{2}^{h}).
$$ 
\item $Z_{2}$ given by $(t_{1},t_{2})=(-2,-2)$. Analogously
$$
e(Z_{2}) = (q-2)e(\oX_{3}^k)e(\oX_{3}^{h})+e(\oX_{3}^k)e(\oX_{1}^{h})+e(\oX_{1}^k)e(\oX_{3}^{h}).
$$
\item $Z_{3}$ given by $(t_{1},t_{2})=(2,-2),(-2,2)$. Now $c\neq 0$. The action of $U\cong \CC$ allows to fix 
$d=0$. Both $\nu,\delta$ are of Jordan type. So we obtain
 $$
 e(Z_{3})=q(e(\oX^k_{2})e(\oX_{3}^{h})+e(\oX^k_{3})e(\oX_{2}^{h})).
 $$
\item $Z_{4}$ given by the subcases:
\begin{itemize}
\item $Z_{4,1}$ given by $(2,t_{2}), t_{2}\neq \pm 2$ and $(-2,t_{2}), t_{2}\neq \pm 2$. We focus on the first case. 
It must be $c\neq 0$. The group $U\cong \CC$ acts freely on the matrix $\nu$, which is of Jordan type. Note also
that $\delta$ is diagonalizable. The second case is similar with $\nu\sim J_-$. We thus get
$$
e(Z_{4,1})= q(e(\oX_{2}^{h})+e(\oX_{3}^{h}))e(\oX_{4}^k/\ZZ_{2}).
$$
\item $Z_{4,2}$ given by $(t_{1},2),t_{1}\neq \pm 2$ and $(t_{1},-2), t_{1}\neq \pm 2$. It is completely similar, interchanging the roles of $\nu$ and $\delta$.
$$
e(Z_{4,2}) =q(e(\oX_{2}^k)+e(\oX_{3}^k))e(\oX_{4}^{h}/\ZZ_{2}).
$$
\end{itemize}
\item $Z_{5}$ corresponding to $t_{1}=t_{2}\neq \pm 2$. Now 
 $$
  \eta=\left(\begin{array}{cc}
      a & b\\
      0 & a^{-1}
    \end{array}\right), \ \delta=\left(\begin{array}{cc}
      a & b+a^{-1}\\
      0 & a^{-1}
    \end{array}\right),
    $$
Therefore $e(Z_5)=q\, e(\overline{Z}_5)$, where $\overline{Z}_{5}$ is a fibration over $a\in \CC- \{0,\pm 1 \}$ whose fibers are 
$\oX_{4,a}^k\times \oX_{4,a}^{h}$. Thus the
 Hodge monodromy representation is given in (\ref{eqn:tres}),
\begin{align*}
R(\overline{Z}_{5}/\ZZ_2) & =A T+ BS_2+CS_{-2}+DS_0 ,\\
R(\overline{Z}_{5}) & =(A+D)T+(B+C)N, \\
 e(Z_5)  &= q\, e(\overline{Z}_{5})=q((q-3)(A+D)-2(B+C)).
\end{align*}

\item $Z_{6}$ corresponding to the open stratum $t_{1},t_{2}\neq \pm 2, t_{1}\neq t_2$. 
As $c\neq 0$, we can arrange $d=0$ by using the action of $U\cong \CC$. Both 
$\delta$ and $\nu$ are diagonalizable matrices. If we ignore for a while the condition 
$t_{1}\neq t_{2}$, the total space is isomorphic to $\CC \times \oX_{4}^k/\ZZ_2 \times \oX_{4}^{h}/\ZZ_2$. 
The fibration over the diagonal $(t_{1},t_{1})$ has total space isomorphic $\CC \times (\overline{Z}_{5}/\ZZ_{2})$. Thus
\begin{align*}
e(\overline{Z}_{5}/\ZZ_2) &=(q-2)A-(B+C+D), \\
 e(Z_{6}) &=q(e(\oX_{4}^k/\ZZ_2)e(\oX_{4}^{h}/\ZZ_2)-e(\overline{Z}_{5}/\ZZ_2)).
\end{align*}
\end{itemize}

Adding all up,
\begin{align*}
e(\oX_{2}^{k+h}) = & \,e(\oX_{2}^k)e(\oX_{0}^{h})+e(\oX_{0}^k)e(\oX_{2}^{h})-2e(\oX^k_{2})e(\oX_{2}^{h})
+e(\oX_{3}^k)e(\oX_{1}^{h})+e(\oX_{1}^k)e(\oX_{3}^{h})-2e(\oX_{3}^k)e(\oX_{3}^{h}) \\
& +q\, (e(\oX^k_{2})+e(\oX^k_{3})+e(\oX_{4}^k/\ZZ_{2}))( e(\oX_{2}^{h})+ e(\oX_{3}^{h})+e(\oX_{4}^{h}/\ZZ_{2}))
+q\,( e(\overline{Z}_{5})- e(\overline{Z}_{5}/\ZZ_2)).
\end{align*}

Setting $k=g-1$, $h=1$, and substituting the values $A,B,C,D$ from (\ref{eqn:tres.5})
and the values of $e(\oX_j^1)$ and $a_1,b_1,c_1,d_1$ from Section \ref{sec:strata}, we have
 \begin{align}\label{eqn:eX2.} \tag{$\gamma$}
  e_2^g=e(\oX_2^{g}) = & \, 
 (q^3 - 2 q^2 - 3 q)e_0^{g-1}+ (q^3 + 3 q^2) e_1^{g-1} \nonumber \\ 
&+  (q^5+q^4+ 3 q^2 + 3 q)e_2^{g-1}+ (q^5-3 q^3- 6 q^2  ) e_3^{g-1}  \nonumber \\ 
& +(q^6-2q^5-3q^4+q^3+3q^2)a_{g-1}+(-q^5+2q^4-4q^3+3q^2)  b_{g-1}\nonumber \\
& +(-q^5-q^4-4q^3+6q^2) c_{g-1}+(-2q^4-q^3+3q^2)d_{g-1} \nonumber
 \end{align}

%%%%%%%%%%%%%%%%%%%%%%%%%%%%%%%%%%%%%%%%%%%%%%%%%%%%
\section{Computation of $e(\oX_3^{k+h})$}\label{subsec:3}
%%%%%%%%%%%%%%%%%%%%%%%%%%%%%%%%%%%%%%%%%%%%%%%%%%%%

This is similar to the previous case. Consider
$$
Z'=\lbrace (A_{1},B_{1},\ldots, A_{k+h},B_{k+h}) \mid \prod_{i=1}^{k}[A_{i},B_{i}]= J_{-}\prod_{i=1}^{h}[B_{k+i},A_{k+i}] \rbrace.
$$
We write
$$
\nu =\prod_{i=1}^{h}[B_{k+i},A_{k+i}] =\begin{pmatrix} a & b \\ c & d \end{pmatrix}, \qquad 
\delta = \prod_{i=1}^{k}[A_{i},B_{i}] = J_{-}\nu = \begin{pmatrix} -a+c & -b+d \\ -c & -d  \end{pmatrix}.
$$
Let $t_{1}=\tr \nu$, $t_{2}=\tr \delta$. In this case, $c=t_{2}+t_{1}$. Stratifying as in Section \ref{subsec:2}, we get

\begin{itemize}
\item $Z'_{1}$ given by $(t_{1},t_{2})=(2,-2)$. We obtain
$$
e(Z'_{1})=(q-2)e(\oX^h_{2})e(\oX_{3}^{k})+e(\oX_{2}^h)e(\oX_{1}^{k})+e(\oX_{0}^h)e(\oX_{3}^{k}).
$$
\item $Z'_{2}$ given by $(t_{1},t_{2})=(-2,2)$. Analogously
$$
e(Z'_{2})=(q-2)e(\oX^h_{3})e(\oX_{2}^{k}) +e(\oX^h_{3})e(\oX_{0}^{k})+e(\oX^h_{1})e(\oX_{2}^{k}).
$$
\item $Z'_{3}$ given by the two values $(t_{1},t_{2})=(2,2),(-2,-2)$. We get
$$
e(Z'_{3})=q(e(\oX^k_{2})e(\oX_{2}^{h})+e(\oX_{3}^k)e(\oX_{3}^{h})).
$$
\item $Z'_{4}$, divided into two possible cases:
\begin{itemize}
\item $Z'_{4,1}$ given by the lines $t_{1}=2$, $t_{2}\neq \pm 2$ and $t_{1}=-2$, $t_{2}\neq \pm 2$. We get
$$
e(Z'_{4,1})=q(e(\oX^h_{2})+e(\oX^h_{3}))e(\oX_{4}^{k}/\ZZ_2).
$$
\item $Z'_{4,2}$  given by the lines $t_{2}=2$, $t_{1}\neq \pm 2$ and $t_{2}=-2$, $t_{1}\neq \pm 2$. We obtain
$$
e(Z'_{4,2})=q(e(\oX_{2}^{k})+e(\oX_{3}^{k}))e(\oX_{4}^h/\ZZ_2).
$$
\end{itemize}
\item $Z'_{5}$ corresponding to $t_{1}=-t_{2}\neq \pm 2$. So $c=0$ and 
 $$
  \eta=\left(\begin{array}{cc}
      a & b\\
      0 & a^{-1}
    \end{array}\right), \ \delta=\left(\begin{array}{cc}
      -a & -b+a^{-1}\\
      0 & -a^{-1}
    \end{array}\right).
    $$
Therefore $e(Z'_5)=q\, e(\overline{Z}{}'_5)$, where $\overline{Z}{}'_{5}$ is a fibration over 
$a\in \CC- \{0,\pm 1 \}$ whose fibers are 
$\oX_{4,a}^k\times \oX_{4,-a}^{h}$. Thus the Hodge monodromy representation is given in (\ref{eqn:tres'}),
\begin{align*}
R(\overline{Z}{}'_{5}/\ZZ_2) & =A' T+ B'S_2+C'S_{-2}+D'S_0 ,\\
R(\overline{Z}{}'_{5}) & =(A'+D')T+(B'+C')N, \\
 e(Z'_5)  &= q\, e(\overline{Z}{}'_{5})=q((q-3)(A'+D')-2(B'+C')).
\end{align*}

\item $Z'_{6}$ corresponding to the open stratum $t_{1},t_{2}\neq \pm 2$, $t_{1}\neq -t_{2}$. The action of $U\cong \CC$ can be used to set $d=0$.
The total space, ignoring the condition $t_{1}\neq -t_{2}$, gives a contribution of $e(\oX_{4}^k)e(\oX_{4}^{h})$. 
The fibration over the diagonal $(t_{1},-t_{1})$ has total space isomorphic 
$\CC \times (\overline{Z}{}'_{5}/\ZZ_{2})$. Thus
\begin{align*}
e(\overline{Z}{}'_{5}/\ZZ_2) &=(q-2)A'-(B'+C'+D') ,\\
 e(Z'_{6}) &=q(e(\oX_{4}^k/\ZZ_2)e(\oX_{4}^{h}/\ZZ_2)-e(\overline{Z}'_{5}/\ZZ_2)).
\end{align*}
\end{itemize}

Adding all up,
\begin{align*}
e(\oX_{2}^{k+h}) = & \,e(\oX_{2}^k)e(\oX_{1}^{h})+e(\oX_{0}^k)e(\oX_{3}^{h})-2e(\oX^k_{2})e(\oX_{2}^{h})
+e(\oX_{3}^k)e(\oX_{0}^{h})+e(\oX_{1}^k)e(\oX_{2}^{h})-2e(\oX_{2}^k)e(\oX_{3}^{h}) \\
& +q\, (e(\oX^k_{2})+e(\oX^k_{3})+e(\oX_{4}^k/\ZZ_{2}))( e(\oX_{2}^{h})+ e(\oX_{3}^{h})+e(\oX_{4}^{h}/\ZZ_{2}))
+q\,( e(\overline{Z}{}'_{5})- e(\overline{Z}{}'_{5}/\ZZ_2)).
\end{align*}

Setting $k=g-1$, $h=1$, and substituting the values $A',B',C',D'$ from (\ref{eqn:tres'.5})
and the values of $e(\oX_j^1)$ and $a_1,b_1,c_1,d_1$ from Section \ref{sec:strata}, we have:
 \begin{align}\label{eqn:eX2.} \tag{$\eta$}
  e_3^g=e(\oX_3^{g}) = & \, 
 (q^3 +3 q^2 )e_0^{g-1}+ (q^3 -2 q^2-3q) e_1^{g-1} \nonumber \\ 
&+ (q^5-3 q^3- 6 q^2  ) e_2^{g-1}+  (q^5+q^4+ 3 q^2 + 3 q)e_3^{g-1}  \nonumber \\ 
& + (q^{6}-2q^5-3q^4+q^3+3q^2)a_{g-1} + (-q^5-q^4-4q^3+6q^2) b_{g-1} \nonumber \\
& +(-q^5+2q^4-4q^3+3q^2) c_{g-1}+(-2q^4-q^3+3q^2) d_{g-1}\nonumber 
 \end{align}

%
%
%\begin{rem}
%Note the following: if we write $A,B,C,D,E,F,G,H$ for the coefficients of $e_{2}^{g},e_{3}^{g}$, there is a certain symmetry between both Jordan cases
%\begin{align*}
%e(\oX_{2})^{g} &=Ae(X_{0}^{g-1})+Be(X_{1}^{g-1})+Ce(\oX_{2}^{g-1})+De(\oX_{3}^{g-1})+Ea_{g-1}+Fb_{g-1}+Gc_{g-1}+Hd_{g-1} \\
%e(\oX_{3})^{g} & =Be(X_{0}^{g-1})+Ae(X_{1}^{g-1})+De(\oX_{2}^{g-1})+Ce(\oX_{3}^{g-1})+Ea_{g-1}+Gb_{g-1}+Fc_{g-1}+Hd_{g-1} 
%\end{align*}
%\end{rem}

%%%%%%%%%%%%%%%%%%%%%%%%%%%%%%%%%%%%%%%%%%%%%%%%%%%%
\section{Computation of $R(\oX_4^{k+h})$}\label{subsec:4}
%%%%%%%%%%%%%%%%%%%%%%%%%%%%%%%%%%%%%%%%%%%%%%%%%%%%

Now we move to the stratum $\oX_4^g$. This one is controlled by a Hodge monodromy representation
$R(\oX_4^g/\ZZ_2)$. We start by computing $R(\oX_4^g)$. 
As before, we write
$$
\oX_{4}^{k+h}= \lbrace (A_{1},B_{1},\ldots, A_{k+h},B_{k+h},\lambda) \mid
 \prod_{i=1}^{k}[A_{i},B_{i}]= \xi_\lambda \prod_{i=1}^{h}[B_{k+i},A_{k+i}] , \lambda\neq 0, \pm 1 \},
$$
where $\xi_\lambda=\begin{pmatrix} \lambda & 0 \\ 0 & \lambda^{-1} \end{pmatrix}$.
We are going to study the fibration $\oX_{4}^{k+h}\to \CC-\{0,\pm 1\}$, with fiber
$\oX_{4,\lambda}^{k+h}$.
 Let
$$
\nu =\prod_{i=1}^{h}[B_{k+i},A_{k+i}] =\begin{pmatrix} a & b \\ c & d \end{pmatrix}, \qquad 
\delta = \prod_{i=1}^{k}[A_{i},B_{i}] = \xi_\lambda \nu =\begin{pmatrix} \lambda a & \lambda b \\ \lambda^{-1} c & \lambda^{-1}d \end{pmatrix}.
$$
Note that $t_{1}=\tr \nu$, $t_{2}= \tr \delta$ and $\lambda$ determine $a,d$, and $bc=ad-1$.

We follow the stratification in terms of the traces $(t_{1},t_{2})$ given in \cite[Section 10]{lomune} for the genus $2$ case.
We decompose $\oX_{4,\lambda}=\bigsqcup_{j=1}^7 Z_{j,\lambda}$, where
\begin{itemize}
\item $Z_{1,\lambda}$ corresponding to $t_{1}=\pm 2, t_{2}=\pm 2$. In this case both $\nu,\delta$ are of Jordan type. 
Focus on the case $(t_1,t_2)=(2,2)$, the other cases being similar. Taking an adequate basis,
$$
\nu=\begin{pmatrix} 1 & x \\ 0 & 1\end{pmatrix}, \quad \delta=\begin{pmatrix} 1 & 0 \\ y & 1\end{pmatrix},
$$
for certain $x,y\in \CC^*$. We can fix $x=1$ by rescaling the basis vectors.
Since $\delta\nu^{-1}=\begin{pmatrix} \lambda & 0 \\ 0 & \lambda^{-1} \end{pmatrix}$, we obtain that 
$\lambda+\lambda^{-1}=2-xy$ , so $y$ is also fixed. When varying $\lambda$, we see that there is no monodromy around the punctures. 
Therefore, taking also care of all four possibilities for $(t_1,t_2)$, we have
\begin{align*}
R(Z_{1}) & =e(Z_{1,\lambda})T \\ &=
(q-1)(e(\oX_{2}^k)e(\oX_{2}^{h})+e(\oX_{2}^k)e(\oX_{3}^{h})
+e(\oX_{3}^k)e(\oX_{2}^{h})+e(\oX_{3}^k)e(\oX_{3}^{h}))T \\
&=(q-1)(e(\oX_{2}^k)+e(\oX_{3}^k))(e(\oX_{2}^{h})+e(\oX_{3}^{h}))T ,
\end{align*}
where $T$ is the trivial representation.

\item $Z_{2,\lambda}$ corresponding to $(t_{1},t_2)=(2, \lambda+\lambda^{-1})$ and $(t_{1},t_{2})=(-2,-\lambda-\lambda^{-1})$. We focus on the first case. 
In this situation $bc=0$, so there are three possibilities: either $b=c=0$ (in which case $\nu=\Id$) or $b=0,c\neq 0$ or $b\neq 0,c=0$ (in either case 
there is a parameter in $\CC^*$ and $\nu\sim J_+$). In all cases, there is no monodromy for $\nu$ as $\lambda$ moves in 
$\CC-\{0,\pm 1\}$. On the other hand, 
$\delta\sim \begin{pmatrix} \lambda & 0 \\ 0 & \lambda^{-1} \end{pmatrix}$ gives a contribution $R(\oX^{h}_{4})$. 
The second case is analogous, changing $\Id$ by $-\Id$ and $J_{+}$ by $J_{-}$. Therefore
\begin{align*}
R(Z_{2}) & =  (e(\oX^k_{0})+2(q-1)e(\oX^k_{2})+e( \oX^k_{1})+2(q-1)e(\oX^k_{3}))R(\oX_{4}^{h}).
\end{align*}

\item $Z_{3,\lambda}$ corresponding to $(t_{1},t_2)=(\lambda+\lambda^{-1},2)$ and $(t_{1},t_{2})=(-\lambda-\lambda^{-1},-2)$. 
This is completely analogous to the previous case, so
\begin{align*}
R(Z_{3}) & = (e(\oX^h_{0})+2(q-1)e(\oX^h_{2})+e( \oX^h_{1})+2(q-1)e(\oX^h_{3}))R(\oX_{4}^{k}).
\end{align*}

\item $Z_{4,\lambda}$ defined by $t_{1}=2, t_{2}\neq \pm 2, \lambda+\lambda^{-1}$ and $t_{1}=-2, t_{2}\neq \pm 2, -\lambda-\lambda^{-1}$. 
For each $\lambda$, $(t_{1},t_{2})$ move in (two) punctured lines $\{ (t_{1},t_{2}) \mid t_{1}=\pm 2, t_{2}\neq \pm 2, \pm (\lambda+\lambda^{-1}) \}$, 
where $\nu$  is of Jordan
type and $\delta$ is of diagonal type. Both families can be trivialized, giving a contribution of 
$e(\oX_{2}^k)$ times $e(\oX_{4}^{h}/\mathbb{Z}_{2})$ for one line, and
$e(\oX_{3}^k)$ times $e(\oX^{h}_{4}/\ZZ_2)$ for the other line. The missing fiber $\oX_{4,\lambda}^{h}$ over $\lambda+\lambda^{-1}$, 
which needs to be removed, has monodromy given by $R(\oX_{4}^{h})$ as $\lambda$ varies. Finally, there is a $(q-1)$ factor due to the fact that $bc\neq 0$. Therefore
\begin{align*}
R(Z_{4}) = & (q-1)(e(\oX_{2}^k)+e(\oX_{3}^k))(e(\oX_{4}^{h}/\mathbb{Z}_{2})T-R(\oX^{h}_{4})) .
\end{align*}
 
\item $Z_{5,\lambda}$ defined by $t_{2}=2, t_{1}\neq \pm 2, \lambda +\lambda^{-1}$ and $t_{2}=-2, t_{1}\neq \pm 2, 
-\lambda-\lambda^{-1}$. Similarly to $Z_{4,\lambda}$, we obtain
\begin{align*}
R(Z_{5}) & = (q-1)(e(\oX_{2}^h)+e(\oX_{3}^h))(e(\oX_{4}^{k}/\mathbb{Z}_{2})T-R(\oX^{k}_{4})) .
\end{align*}

\item $Z_{6,\lambda}$. This stratum corresponds to the set $\{ (t_{1},t_{2})\mid t_{1},t_{2}\neq \pm 2, ad=1 \}$, 
which is a hyperbola $H_\lambda$ for every $\lambda$ (see \cite[Figure 1, Section 10]{lomune}).
There is a contribution of $2q-1$ which accounts for  $bc=0$. Parametrizing 
$H_\lambda$ by $\mu \in \CC^* -\{ \pm 1, \pm \lambda^{-1} \}$ as in \cite[Section 10]{lomune}, 
we obtain a fibration over $\CC^* -\{ \pm 1, \pm \lambda^{-1} \}$ 
whose fiber over $\mu$ is $\oX^k_{4,\mu}\times \oX^{h}_{4,\lambda\mu}$, for each $\lambda$. 
When $\lambda$ varies over $\CC -\{0, \pm 1 \}$, 
we can extend the local system trivially to the cases $\lambda,\mu = \pm 1$. This extension 
can be regarded as a local system over the set of $(\lambda,\mu) \in \CC^*\times \CC^*$,
 $$
 \overline{Z}_{6}=\oX^k_{4}\times m^{*}\oX^{h}_{4} \longrightarrow \CC^*\times\CC^*\, , 
 $$
where $m: \CC^*\times \CC^* \rightarrow \CC^*$ maps $(\lambda,\mu)\mapsto \lambda\mu$. 
The Hodge monodromy representation of $\overline{Z}_{6}$ belongs to $R(\zz)[q]$ (with generators 
$N_{1},N_{2}$ denoting the  representation which is not trivial over the generator of the 
fundamental group of the first and second copies of $\CC^*$, respectively, and
$N_{12}=N_1\otimes N_2$). So we get
\begin{align*}
R_{ \CC^*\times\CC^*} (\overline{Z}_{6}) & =((a_k+d_k)T+(b_k+c_k)N_{2})\otimes  ((a_h+d_h)T+(b_h+c_h)N_{12})  \\
& =(a_h+d_h)(a_k+d_k) T +(b_h+c_h)(b_k+c_k) N_{1} +(a_h+d_h) (b_k+c_k)  N_{2}+ (b_h+c_h)(a_k+d_k)N_{12}.
\end{align*}
To obtain the Hodge monodromy representation over $\lambda \in \CC^*$, we use the projection 
$\pi_{1}:\CC^*\times\CC^* \rightarrow \CC^*$, $(\lambda,\mu)\mapsto \lambda$, which maps
$T\mapsto e(T) T=(q-1)T$, $N_2 \mapsto e(N_2)T=0$, $N_1\mapsto e(T) N=(q-1)N$, 
$N_{12} \mapsto e(N_2)N=0$ for the representations. Therefore $R_{\CC^*}(\overline{Z}_6)=
(q-1)((a_h+d_h)(a_k+d_k) T +(b_h+c_h)(b_k+c_k) N)$. Now we have to substract the contribution 
from the sets $\mu=\pm1,\pm\lambda^{-1}$. The first two yield $-2e(\oX_4^k)R(\oX_4^h)$ and
the second two yield $-2e(\oX_4^h)R(\oX_4^k)$. Therefore
\begin{align*}
R(\overline{Z}_{6}) = &(q-1)((a_h+d_h)(a_k+d_k) T +(b_h+c_h)(b_k+c_k) N)-2e(\oX_{4,\lambda}^k)R(\oX_4^h)-2e(\oX_{4,\lambda}^h)R(\oX_4^k),\\
R(Z_6) = &(2q-1) R(\overline{Z}_{6}) .
\end{align*}

\item $Z_{7,\lambda}$ corresponding to the open stratum given by the set of $(t_{1},t_{2})$ such that $t_{i} \neq \pm 2$, 
$i=1,2$ and $(t_{1},t_{2})\not\in H_\lambda$. If we forget about the condition $(t_{1},t_{2})\in H_\lambda$, $Z_{7,\lambda}$ is a fibration over 
$(t_{1},t_{2})$ with fiber isomorphic to $\oX^h_{4,\mu_{1}}\times \oX^{k}_{4,\mu_{2}}$, $t_{i}=\mu_i+\mu^{-1}_i$, $i=1,2$. 
Its monodromy is trivial, as the local system is trivial when $\lambda$ varies. The contribution over $H_\lambda$, already computed in the previous
stratum, is $R(\overline{Z}_6)$. So we get
\begin{align*}
R(Z_{7}) =& (q-1)(e(\oX^k_{4}/\mathbb{Z}_{2})e(\oX^h_{4}/\ZZ_2)T-R(\overline{Z}_{6})) .
\end{align*}
\end{itemize}

Adding all the pieces, we get
\begin{align*}
R(\oX^{k+h}_{4}) =& 
 (q-1)(e(\oX_{2}^k)+e(\oX_{3}^k)+e(\oX^k_{4}/\mathbb{Z}_{2}))(e(\oX_{2}^{h})+e(\oX_{3}^{h})+e(\oX^h_{4}/\mathbb{Z}_{2}))T\nonumber \\
 &+ (e(\oX^k_{0})+e( \oX^k_{1})+(q-1)e(\oX^k_{2})+(q-1)e(\oX^k_{3}))R(\oX_{4}^{h}) \\
& +(e(\oX^h_{0})+e( \oX^h_{1})+(q-1)e(\oX^h_{2})+(q-1)e(\oX^h_{3}))R(\oX_{4}^{k}) + q R(\overline{Z}_{6}). \nonumber 
\end{align*}

Setting $k=g-1$, $h=1$, we have:
 \begin{align} \label{eqn:R4}
R(\oX^{g}_{4}) = & \, \Big(
(q^3-1)e_0^{g-1}+(q^3-1)e_1^{g-1}+(q^5-3q^3+2q^2)e_2^{g-1}+(q^5-3q^3+2q^2)e_3^{g-1} \nonumber \\
&+(q^6- 2q^5 -2q^4 + 4q^3- 3 q^2 +2)a_{g-1}
+(- q^5 -q^4 + 2q^3- 2 q^2 +q+1)(b_{g-1}+c_{g-1}) \nonumber\\ 
&+(-q^4-2q^2+2q+1)d_{g-1} \Big)T \\
&+ \Big( (3q^2-3q)e_0^{g-1}+(3q^2-3q)e_1^{g-1}+(3  q^3 - 6 q^2 + 3 q)e_2^{g-1}+(3  q^3 - 6 q^2 + 3 q)e_3^{g-1} \nonumber\\
&+(-6q^3+6q^2)(a_{g-1}+d_{g-1}) +(  4q^4-14q^3+ 10  q^2) (b_{g-1}+c_{g-1}) \Big)N \nonumber
\end{align}

%%%%%%%%%%%%%%%%%%%%%%%%%%%%%%%%%%%%%%%%%%%%%%%%%%%%%
\section{Computation of $R(\oX_4^{g}/\ZZ_2)$}\label{subsec:5}
%%%%%%%%%%%%%%%%%%%%%%%%%%%%%%%%%%%%%%%%%%%%%%%%%%%%%

%Our intention is now to implement the formulas obtained in Section \ref{sec:3} to find the polynomials $e(\oX_0^g)$,  $e(\oX_1^g)$,  $e(\oX_2^g)$, 
% $e(\oX_3^g)$, $a_g,b_g,c_g$ and $d_g$, inductively. Assume that $g\geq 4$, since the cases $g=1,2$ are covered in 
%\cite{lomune} and $g=3$ in  \cite{mamu}.
%
%Suppose that we know the polynomials $e(\oX_0^k)$,  $e(\oX_1^k)$,  $e(\oX_2^k)$, 
% $e(\oX_3^k)$, and that $R(\oX_4^k/\ZZ_2)=a_k T+b_k S_2+ c_k S_{-2}+ d_k S_0$, 
%for all $k<g$. We have the following result.

\begin{lem} \label{lem:RR}
Suppose that  $R(\oX_4^k/\ZZ_2)=a_k T+b_k S_2+ c_k S_{-2}+ d_k S_0$, for all $k<g$.
 Then the Hodge monodromy representation $R(\oX^g_{4}/\mathbb{Z}_{2})$ is of the form
$R(\oX^g_{4}/\mathbb{Z}_{2})= a_gT+b_gS_{2}+c_gS_{-2}+d_gS_{0}$, for some polynomials $a_g,b_g,c_g,d_g\in \ZZ[q]$.
\end{lem}

\begin{proof}
 The Hodge monodromy representation $R(\oX^g_{4}/\mathbb{Z}_{2})$ lies in the representation ring of the
 fundamental group of $\CC-\{\pm 2\}$. Under the double cover $\CC-\{0,\pm 1\} \to \CC-\{ \pm 2\}$, it reduces
to $R(\oX^g_{4})$. By Section \ref{subsec:4},  $R(\oX^g_{4})$ is of order $2$. Hence
$R(\oX^g_{4}/\mathbb{Z}_{2})$ has only monodromy of order $2$ over the loops $\gamma_{\pm 2}$ 
around the points $\pm 2$. This is the statement of the lemma.
\end{proof}

\begin{prop} \label{prop:balanced}
All $\oX_0^g$, $\oX_1^g$, $\oX_2^g$, $\oX_3^g$, $\oX_4^g$ and $\oX_{4,\lambda}^g$ are of balanced type.
\end{prop}

\begin{proof}
 By \cite[Proposition 2.8]{lomune}, if $Z= \bigsqcup Z_i$ and all $Z_i$ are of balanced type, then $Z$ is of balanced type.
Also, if $\ZZ_2$ acts on $Z$ and $Z$ is of balanced type, so is $Z/\ZZ_2$. Also Theorem  \ref{thm:general-fibr} says
that if $F\to Z \to B$ is a fibration with $F$ of balanced type, with
either $B=\CC-\{0,\pm 1\}$ and Hodge monodromy $R(Z)=a T+ b N$ or 
$B=\CC-\{\pm 2\}$ and Hodge monodromy $R(Z)=a T+ b S_2+cS_{-2}+dS_0$,  then $Z$ is of balanced type.

In \cite{lomune} it is proved that the result holds for $g=1,2$. Also $\SL(2,\CC)$, $\PGL(2,\CC)$, $\PGL(2,\CC)/D$,
$\PGL(2,\CC)/U$ are of balanced type. Now assume that all 
 $\oX_0^k$, $\oX_1^k$, $\oX_2^k$, $\oX_3^k$, $\oX_4^k$ and $\oX_{4,\lambda}^k$ are of balanced type
for $k<g$. A look at the description of all strata for which we compute the E-polynomials
 in Sections \ref{subsec:0}--\ref{subsec:3} convinces us that $\oX_0^g$, $\oX_1^g$, $\oX_2^g$, $\oX_3^g$
are of balanced type. The same is true for $\oX_{4,\lambda}^g$ by the stratification in Section \ref{subsec:4}. 
Finally formula (\ref{eqn:R4}) gives us that $\oX_4^g$ is also of balanced type.
\end{proof}

\medskip

Now to find the four polynomials  $a_g,b_g,c_g,d_g\in \ZZ[q]$, we need four equations.
Two come from the fact that $R(\oX^g_{4})=(a_g+d_g)T+(b_g+c_g)N$. From (\ref{eqn:R4}), we have
\begin{align}\label{eqn:aaa}
 a_g+d_g =&
(q^3-1)e_0^{g-1}+(q^3-1)e_1^{g-1}+(q^5-3q^3+2q^2)e_2^{g-1}+(q^5-3q^3+2q^2)e_3^{g-1} \\
&+(q^6- 2 q^5 -2 q^4 + 4 q^3- 3 q^2 +2)a_{g-1}
+(-q^5-q^4+2q^3-2q^2+q+1) (b_{g-1}+c_{g-1})  \nonumber\\ &+(-q^4-2q^2+2q+1)d_{g-1}  \nonumber\\
 b_g+d_g =& (3q^2-3q)e_0^{g-1}+(3q^2-3q)e_1^{g-1}+(3  q^3 - 6 q^2 + 3 q)e_2^{g-1}+(3  q^3 - 6 q^2 + 3 q)e_3^{g-1} 
\label{eqn:bbb}\\
&+(-6q^3+6q^2)(a_{g-1}+d_{g-1}) +(  4q^4-14q^3+ 10  q^2) (b_{g-1}+c_{g-1}).\nonumber
\end{align}

One more equation is obtained by computing $e(\oX_0^{g+1})$ with $k=g$, $h=1$ and with $k=g-1$, $h=2$, and
equating. From ($\alpha$) with $k=g$, we get
 \begin{align}\label{eqn:cc}
 e_0^{g+1}= & \, (q^4 + 4 q^3 - q^2 - 4 q) e_0^{g}+ (q^3-q) e_1^{g} \nonumber \\ 
&+ ( q^5- 2 q^4 - 4 q^3+ 2 q^2  +3 q ) e_2^{g}+ (q^5+3q^4-q^3-3q^2) e_3^{g}  \\ 
&+ (q^6-2q^5-4q^4+3q^2+2q) a_{g} + (-q^5-4q^4+4q^2+q)b_{g}   \nonumber\\
&+ (2q^5-7q^4-3q^3+7q^2+q) c_{g} + (-5q^4-q^3+5q^2+q) d_{g} \nonumber
 \end{align}
Using (\ref{eqn:eX0}) with $k=g-1$, $h=2$, and the 
 values of $e(\oX_j^2)$ and $a_2,b_2,c_2,d_2$ from Section \ref{sec:strata}, we have
 \begin{align}\label{eqn:cc'}
 e_0^{g+1}= & \, (q^9+ q^8+12q^7+2q^6-3q^4-12q^3-q) e_0^{g-1}
+ (q^9-3q^7-30q^6+30q^4+3q^3-q) e_1^{g-1}\nonumber \\ 
&+ ( q^{11}-4q^9-4q^8-36q^7+24q^5+4q^4+15q^3)e_2^{g-1} \nonumber \\ 
&+ (q^{11}-4q^9+15q^8+9q^7+30q^6-6q^5-45q^4) e_3^{g-1}  \nonumber \\ 
&+ (q^{12}- 2 q^{11}-  4 q^{10}+ 6 q^9 - 6 q^8+ 18 q^7 - 6 q^6 - 18 q^5+ 15 q^4- 6 q^3+2 q ) a_{g-1}  \\
&+ (- q^{11}- q^{10}+ 3 q^9 - 42 q^8+ 54 q^7 + 30 q^6- 54 q^5 + 12 q^4 - 3 q^3 + q^2 +q )b_{g-1} \nonumber \\
&+ (- q^{11}- q^{10}+ 18 q^9 - 27 q^8 + 24 q^7- 39 q^5+ 27 q^4- 3 q^3+ q^2 +q  ) c_{g-1} \nonumber\\
&+ (- 2 q^{10}- 9 q^8 + 24 q^7 - 21 q^5+ 9 q^4  - 4 q^3+ 2 q^2 +q  ) d_{g-1} \nonumber
 \end{align}

A fourth equation are obtained by computing $e(\oX_1^{g+1})$ with $k=g$, $h=1$ and with $k=g-1$, $h=2$, and
equating. From ($\beta$) with $k=g$, we get
 \begin{align}\label{eqn:dd}
  e_1^{g+1}=& \,  (q^3-q)e_0^{g}+ (q^4 + 4 q^3 - q^2 - 4 q) e_1^{g} \nonumber\\ 
&+  (q^5+3q^4-q^3-3q^2)e_2^{g}+ ( q^5- 2 q^4 - 4 q^3+ 2 q^2  +3 q ) e_3^{g}  \\ 
& +(q^6-2q^5-4q^4+3q^2+2q) a_{g} + (2q^5-7q^4-3q^3+7q^2+q) b_{g} \nonumber\\
&+ (-q^5-4q^4+4q^2+q) c_{g} + (-5q^4-q^3+5q^2+q) d_{g} \nonumber
 \end{align}
Using (\ref{eqn:eX1}) with $k=g-1$, $h=2$, and the 
 values of $e(\oX_j^2)$ and $a_2,b_2,c_2,d_2$ from Section \ref{sec:strata}, we have
 \begin{align}\label{eqn:dd'}
 e_1^{g+1}= & \, (q^9-3q^7-30q^6+30q^4+3q^3-q) e_0^{g-1} +(q^9+ q^8+12q^7+2q^6-3q^4-12q^3-q) e_1^{g-1} \nonumber\\ 
&+ (q^{11}-4q^9+15q^8+9q^7+30q^6-6q^5-45q^4) e_2^{g-1} \nonumber \\ 
&+ (q^{11}-4q^9-4q^8-36q^7+24q^5+4q^4+15q^3)e_3^{g-1} \\
&+ (q^{12}- 2 q^{11}-  4 q^{10}+ 6 q^9 - 6 q^8+ 18 q^7 - 6 q^6 - 18 q^5+ 15 q^4- 6 q^3+2 q ) a_{g-1} \nonumber\\
&+ (- q^{11}- q^{10}+ 18 q^9 - 27 q^8 + 24 q^7- 39 q^5+ 27 q^4- 3 q^3+ q^2 +q  ) c_{g-1} \nonumber\\
&+ (- q^{11}- q^{10}+ 3 q^9 - 42 q^8+ 54 q^7 + 30 q^6- 54 q^5 + 12 q^4 - 3 q^3 + q^2 +q )b_{g-1}   \nonumber\\
&+ (- 2 q^{10}- 9 q^8 + 24 q^7 - 21 q^5+ 9 q^4  - 4 q^3+ 2 q^2 +q  ) d_{g-1} \nonumber
 \end{align}

The solutions to (\ref{eqn:aaa}), (\ref{eqn:bbb}), (\ref{eqn:cc})=(\ref{eqn:cc'}) and (\ref{eqn:dd})=(\ref{eqn:dd'}),
and using the values of ($\alpha$), ($\beta$), ($\gamma$) and ($\delta$), are given by
\begin{align*}
 a_g =& q^3 e_{0}^{g-1} + q^3 e_{1}^{g-1} +(q^5-3q^3) e_{2}^{g-1} + (q^5-3q^3)e_{3}^{g-1} \\
 &  + (q^6-2q^5-2q^4+4q^3+q^2)a_{g-1} + (-q^5-q^4+2q^3)b_{g-1} +(-q^5-q^4+2q^3)c_{g-1} -2q^4 d_{g-1}\\
 b_g =&-3q e_{0}^{g-1} +3q^2 e_{1}^{g-1} +(3q^3+3q) e_{2}^{g-1} -6q^2 e_{3}^{g-1} \\
 & +(-3q^3+3q^2)a_{g-1} + (4q^4-6q^3+4q^2)b_{g-1} + (-8q^3+6q^2)c_{g-1} + (q^2-3q^3+3q^2)d_{g-1} \\
 c_g =& 3q^2 e_{0}^{g-1} -3q e_{1}^{g-1} -6q^2 e_{2}^{g-1} +(3q^3+3q) e_{3}^{g-1}\\
 & (-3q^3+3q^2)a_{g-1} + (-8q^3+6q^2)b_{g-1} + (4q^4-6q^3+4q^2) c_{g-1} + (-3q^3+3q^2) d_{g-1} \\
 d_g =& -e_{0}^{g-1} -e_{1}^{g-1} +2q^{2}e_{2}^{g-1} +2q^{2}e_{3}^{g-1}\\
 & +(-4q^2+2)a_{g-1} + (-2q^2+q+1)b_{g-1} + (-2q^{2}+q+1)c_{g-1} + (q^{4}-2q^{2}+2q+1)d_{g-1}
\end{align*}
We put this together with equations ($\alpha$), ($\beta$), ($\gamma$) and ($\delta$)
\begin{align*}
 e^g_0 =& (q^4 + 4 q^3 - q^2 - 4 q) e_0^{g-1}+ (q^3-q) e_1^{g-1} + ( q^5- 2 q^4 - 4 q^3+ 2 q^2  +3 q ) e_2^{g-1}+ (q^5+3q^4-q^3-3q^2) e_3^{g-1}  \\ 
&+ (q^6-2q^5-4q^4+3q^2+2q) a_{g-1} + (-q^5-4q^4+4q^2+q)b_{g-1} \\
&+ (2q^5-7q^4-3q^3+7q^2-q) c_{g-1} + (-5q^4-q^3+5q^2-q) d_{g-1}\\
 e^g_1  =& (q^3-q)e_0^{g-1}+ (q^4 + 4 q^3 - q^2 - 4 q) e_1^{g-1} +  (q^5+3q^4-q^3-3q^2)e_2^{g-1}+ ( q^5- 2 q^4 - 4 q^3+ 2 q^2  +3 q ) e_3^{g-1}  \\ 
&+(q^6-2q^5-4q^4+3q^2+2q) a_{g-1} + (2q^5-7q^4-3q^3+7q^2+q) b_{g-1} \\
&+ (-q^5-4q^4+4q^2+q) c_{g-1} + (-5q^4-q^3+5q^2+q) d_{g-1} \\
 e^g_2 =& (q^3 - 2 q^2 - 3 q)e_0^{g-1}+ (q^3 + 3 q^2) e_1^{g-1} +  (q^5+q^4+ 3 q^2 + 3 q)e_2^{g-1}+ (q^5-3 q^3- 6 q^2  ) e_3^{g-1}  \\ 
& +(q^6-2q^5-3q^4+q^3+3q^2)a_{g-1}+(-q^5+2q^4-4q^3+3q^2)  b_{g-1} \\
& +(-q^5-q^4-4q^3+6q^2) c_{g-1}+(-2q^4-q^3+3q^2)d_{g-1} \\
 e^g_3=& (q^3 +3 q^2 )e_0^{g-1}+ (q^3 -2 q^2-3q) e_1^{g-1} + (q^5-3 q^3- 6 q^2  ) e_2^{g-1}+  (q^5+q^4+ 3 q^2 + 3 q)e_3^{g-1}   \\ 
& + (q^{6}-2q^5-3q^4+q^3+3q^2)a_{g-1} + (-q^5-q^4-4q^3+6q^2) b_{g-1}  \\
& +(-q^5+2q^4-4q^3+3q^2) c_{g-1}+(-2q^4-q^3+3q^2) d_{g-1}
\end{align*}
Hence there is a $8\x 8$-matrix $M$ such that if we write 
$v_g=( e^g_0, e^g_1, e^g_2, e^g_3,a_g,b_g,c_g,d_g)^t$,
\begin{equation}\label{eqn:matrix}
v_g =M v_{g-1},
\end{equation}
for all $g\geq 3$. $M$ is the following matrix
\begin{equation} \label{matrixg-1-->g}
\scalemath{0.68}{\left( \begin{array}{c@{\hspace{2em}}c@{\hspace{2em}}c@{\hspace{2em}}c@{\hspace{2em}}c@{\hspace{2em}}c@{\hspace{2em}}c@{\hspace{2em}}c}
 q^4+4q^3 & q^3-q & q^5 -2q^4-4q^3 & q^5+3q^4   & q^6 -2q^5 -4q^4 & -q^5-4q^4 & 2q^5 -7q^4 -3q^3 & -5q^4-q^3  \\
 -q^2-4q  &       & +2q^2+3q       & -q^3 -3q^2 & +3q^2 +2q       & +4q^2+q  & +7q^2+q         & +5q^2 +q    \\
& & & & & & & \\ 
q^3 -q & q^4 +4q^3 & q^5+3q^4  & q^5-2q^4-4q^3 & q^6-2q^5-4q^4 & 2q^5-7q^4-3q^3 & -q^5-4q^4 & -5q^4-q^3 \\
       & -q^2-4q   & -q^3-3q^2 & +2q^2+3q      & +3q^2+2q      & +7q^2 +q       & +4q^2+q   &  +5q^2+q \\
 & & & & & & & \\
q^3-2q^2 & q^3+3q^2 & q^5+q^4  & q^5 -3q^3 & q^6-2q^5-3q^4 & -q^5+2q^4  & -q^5 -q^4   & -2q^4 -q^3 \\
-3q      &          & +3q^2+3q & -6q^2     & +q^3+3q^2     & -4q^3+3q^2 & -4q^3 +6q^2 & +3q^2 \\
 & & & & & & & \\
q^3+3q^2 & q^3-2q^2 & q^5-3q^3 & q^5+q^4  & q^6-2q^5-3q^4 & -q^5-q^4   & -q^5 +2q^4  & -2q^4 -q^3 \\
         & -3q      & -6q^2    & +3q^2+3q & +q^3+3q^2     & -4q^3+6q^2 & -4q^3 +3q^2 & +3q^2 \\
 & & & & & & & \\
q^3 & q^3 & q^5-3q^3 & q^5-3q^3 & q^6 -2q^5-2q^4 & -q^5-q^4 & -q^5-q^4 & -2q^4 \\
    &     &          &          & +4q^3+q^2      & +2q^3    & +2q^3    &       \\
 & & & & & & & \\
-3q & 3q^2 & 3q^3+3q & -6q^2 & -3q^3+3q^2 & 4q^4-6q^3 +4q^2 & -8q^3+6q^2 & -3q^3+3q^2 \\
& & & & & & & \\
3q^2 & -3q & -6q^2 & 3q^3+3q & -3q^3+3q^2 & -8q^3+6q^2 & 4q^4-6q^3+4q^2 & -3q^3+3q^2 \\
& & & & & & & \\
-1 & -1 & 2q^2 & 2q^2 & -4q^2+2 & -2q^2+q+1 & -2q^2+q+1 & q^4 -2q^2 \\
   &    &      &      &         &           &           & +2q+1 
\end{array} \right)}.
\end{equation}
The starting vector is given in Section \ref{sec:strata},
$ v_2= (e_0^2,e_1^2,e_2^2,e_3^2, a_2,b_2,c_2,d_2)^t=
( q^9+q^8+12q^7+2q^6-3q^4-12q^3-q, 
q^9-3q^7-30q^6+30q^4+3q^3-q,
q^9-3q^7-4q^6-39q^5-4q^4-15q^3,
q^9-3q^7+15q^6+6q^5+45q^4,
q^9-3q^7+6q^5,-(45q^5+15q^3),15q^6+45q^4,-6q^4+3q^2-1)^t$.
If we write $v_1=(e_0^1,e_1^1,e_2^1,e_3^1, a_1,b_1,c_1,d_1)^t=
( q^{4}+4q^{3}-q^2-4q,q^{3}-q , q^3-2q^2-3 q,q^3 +3 q^2 ,q^{3},-3q,3q^{2},-1)^t$
and  
$$
 v_0=(1,0,0,0,0,0,0,0)^t,
$$
then equation (\ref{eqn:matrix}) holds for all $g\geq 1$. So
  $$
 v_g=M^g v_0 .
 $$

\begin{rem}
As in \cite{lomune}, we can stratify $\sldos^{2g} = \bigsqcup_{i=0}^{4} X_{i}^{g}$, with
\begin{itemize}
\item $X_0^g=\oX_{0}^{g}$, $e(X_0^g)=e_0^g$.
\item $X_1^g=\oX_{1}^{g}$, $e(X_1^g)=e_1^g$.
\item $X_{2}^{g}= \lbrace (A_{1},B_{1},\ldots, A_{g},B_{g}) \mid \prod_{i=1}^{g}[A_{i},B_{i}]\sim J_+\rbrace \cong
(\PGL(2,\CC)/U) \x \oX_2^g$. So $e(X_2^g)=(q^2-1)e_2^g$.
\item $X_{3}^{g}= \lbrace (A_{1},B_{1},\ldots, A_{g},B_{g}) \mid \prod_{i=1}^{g}[A_{i},B_{i}]\sim  J_- \}
\cong  (\PGL(2,\CC)/U) \x \oX_3^g$. So $e(X_3^g)=(q^2-1)e_3^g$.
\item  $X_{4}^{g}= 
\lbrace (A_{1},B_{1},\ldots, A_{g},B_{g}) \mid 
\prod_{i=1}^{g}[A_{i},B_{i}] \sim \xi_\lambda,$ for some $\lambda\in \CC-\{0,\pm 1\}\rbrace$. Here
$X_4^g \cong ( \PGL(2,\CC)/D \x \oX_4^g  )/\ZZ_2$. Using (\ref{eqn:RX4-RX4Z2->eX4}),
we have  $e(X_4^g)=(q^3-2q^2-q)a_g-(q^2+q)(b_g+c_g)-2q d_g$.
\end{itemize}
Therefore it must be
\begin{equation}\label{EpolyStratificationSL2Cg}
(q^3-q)^{2g} =  e_{0}^{g}+e_{1}^{g}+(q^2-1)(e_{2}^{g}+e_{3}^{g})+(q^3-2q^2-q)a_{g}-(q^2+q)(b_{g}+c_{g})-2qd_{g}. 
\end{equation}
We can prove (\ref{EpolyStratificationSL2Cg}) numerically by induction on $g\geq 0$, using 
(\ref{eqn:matrix}). The equation (\ref{EpolyStratificationSL2Cg}) is certainly true for $g=0$. Suppose it holds
for $g-1$ and let $w=(w_{0},\ldots,w_{7})^t=Mv_{g-1}$. Then 
an easy computation gives
\begin{align*}
 w_{0}&+w_{1}+(q^2-1)(w_{2}+w_{3})+(q^3-2q^2-q)w_{4}-(q^2+q)(w_{5}+w_{6})-2qw_{7} \\ & = (q^3-q)^2(e_{0}^{g-1}+e_{1}^{g-1}+(q^2-1)(e_{2}^{g-1}+e_{3}^{g-1})+(q^3-2q^2-q)a_{g-1}-(q^2+q)(b_{g-1}+c_{g-1})-2qd_{g-1}) \\ & = (q^3-q)^{2} (q^3-q)^{2g-2} =(q^3-q)^{2g},
\end{align*}
so equation (\ref{EpolyStratificationSL2Cg}) holds for $v_{g}=w=Mv_{g-1}$.
\end{rem}

We start by proving Corollary \ref{cor:Hausel} using (\ref{matrixg-1-->g}).
\begin{thm}
For every $g\geq 1$, we have
$e(\mathcal{M}_{J_{-}})+(q+1)e(\mathcal{M}_{-\Id})=e(\mathcal{M}_{\xi_{\lambda}})$.
\end{thm}

\begin{proof}
First, $\cM_{-\Id}=\oX_1^g/\PGL(2,\CC)$, so $e(\cM_{-\Id})=e_1^g/(q^3-q)$. Second,
$\cM_{J_+}=\oX_2^g/U$, so $e(\cM_{J_+})=e_3^g/q$. And third,
$\cM_{\xi_\lambda}= \oX_{4,\lambda}^g /D$, so $e(\cM_{\xi_\lambda})=e(\oX_{4,\lambda}^g) /(q-1)=
(a_g+b_g+c_g+d_g)/(q-1)$.

The assertion is thus equivalent to 
\begin{equation} \label{HauselIdentity}
 (q^2-1)e_{3}^{g}+(q+1)e_{1}^{g}=(q^2+q)(a_{g}+b_{g}+c_{g}+d_{g}),
\end{equation}
for all $g\geq 1$. We proceed by induction starting with $g=0$, where it obviously holds.
If we assume that (\ref{HauselIdentity}) holds for $g-1$, then using (\ref{matrixg-1-->g}),
\begin{align*}
(q^2+q) & (a_{g}+ b_{g}+c_{g}+d_{g})-(q^2-1)e_{3}^{g}-(q+1)e_{1}^{g} \\ 
 & = -q^2(q+1)(q-1)^2e_{1}^{g-1} -q^2(q+1)(q-1)^3e_{3}^{g-1} 
+q^3(q+1)(q-1)^2(a_{g-1}+b_{g-1}+c_{g-1}+d_{g-1}) \\
& = q^2(q-1)^2((q^2+q)(a_{g-1}+b_{g-1}+c_{g-1}+d_{g-1})  -(q^2-1)e_{3}^{g-1}-(q+1)e_{1}^{g-1}) = 0,
\end{align*}
by induction hypothesis.
\end{proof}

Since $v_{g}=M^{g}v_{0}$, we can  obtain closed formulas for $e_{0}^{g},e_{1}^{g},e_{2}^{g},e_{3}^{g},a_{g},b_{g},c_{g},d_{g}$. We summarize them in the following
\begin{prop} \label{Polynomialseig}
For all $g\geq 1$,
\begin{align*}
e_{0}^{g} & = (q^3-q)\left( (q^3-q)^{2g-2}+(q^2-1)^{2g-2}-(q^2-q)^{2g-2}+\frac{1}{2}q^{2g-2}(q+2^{2g}-1)\left( (q+1)^{2g-2}+(q-1)^{2g-2} \right) \right) \\
e_{1}^{g} & = (q^{3}-q)\left( (q^3-q)^{2g-2}+(q^2-1)^{2g-2}-2^{2g-1}(q^{2}+q)^{2g-2}+(2^{2g-1}-1)(q^{2}-q)^{2g-2}) \right)\\
e_{2}^{g} & = (q^{3}-q)^{2g-1}+(2^{2g-1}-1)(q^{2}-q)^{2g-1}-2^{2g-1}(q^{2}+q)^{2g-1}+\frac{1}{2}q^{2g-1}(q-1) \left( (q-1)^{2g-1}-(q+1)^{2g-1} \right) \\
e_{3}^{g} & = (q^{3}-q)^{2g-1}+(2^{2g-1}-1)(q^{2}-q)^{2g-1}+2^{2g-1}(q^{2}+q)^{2g-1} \\
a_{g} & = (q^{3}-q)^{2g-1}+\frac{1}{2}q^{2g-1} \left( (q+1)^{2g-1}-(q-1)^{2g-1} \right) \\
b_{g} & = 2^{2g-1}(q^{2}-q)^{2g-1}-2^{2g-1}(q^{2}+q)^{2g-1} +\frac{1}{2}q^{2g-1} \left( (q+1)^{2g-1}-(q-1)^{2g-1} \right) \\
c_{g} & = 2^{2g-1}(q^{2}-q)^{2g-1}+2^{2g-1}(q^{2}+q)^{2g-1}-\frac{1}{2}q^{2g-1} \left( (q+1)^{2g-1}+(q-1)^{2g-1} \right) \\
d_{g} & = (q^{2}-1)^{2g-1}-\frac{1}{2}q^{2g-1} \left( (q+1)^{2g-1} +(q-1)^{2g-1} \right),
\end{align*}
and also
\begin{align*}
e_{4,\xi_{\lambda}}^{g}& = a_{g}+b_{g}+c_{g}+d_{g} =(q^{3}-q)^{2g-1}+(q^{2}-1)^{2g-1}+(2^{2g}-2)(q^{2}-q)^{2g-1}.
\end{align*}
\end{prop}

\begin{proof}
We know that $v_{g}=M^{g}v_{0}$, where $M$ is given in \ref{matrixg-1-->g}. 
There exists a matrix $Q$ with entries in the fraction field of $\mathbb{Z}[q]$ such that $M=QDQ^{-1}$, 
where $D$ is the diagonal matrix
$$
D=\begin{pmatrix}
(q^2-q)^2 & 0 & 0 & 0 & 0 & 0 & 0 & 0 \\
0 & (q^2+q)^2 & 0 & 0 & 0 & 0 & 0 & 0 \\
0 & 0 & 4(q^2-q)^2 & 0 & 0 & 0 & 0 & 0 \\
0 & 0 & 0 & 4(q^2+q)^2 & 0 & 0 & 0 & 0 \\
0 & 0 & 0 & 0 & (q^2-1)^2  & 0 & 0 & 0 \\
0 & 0 & 0 & 0 & 0 & (q^3-q)^2 & 0 & 0 \\
0 & 0 & 0 & 0 & 0 & 0 & (q^2-q)^2 & 0 \\
0 & 0 & 0 & 0 & 0 & 0 & 0 & (q^2+q)^2
\end{pmatrix}
$$
As $M^{g}=QD^{g}Q^{-1}$, a straightforward computation gives the desired formulas.	
\end{proof}

Proposition \ref{Polynomialseig} gives us the E-polynomials of all the moduli spaces where the quotient is geometric dividing by the E-polynomials of the respective stabilizers. We obtain
\begin{thm}
For all $g\geq 1$, 
\begin{align*}
e(\mathcal{M}_{-\Id}) = & \, e_{1}^{g}/(q^3-q) = (q^3-q)^{2g-2}+(q^2-1)^{2g-2}-2^{2g-1}(q^{2}+q)^{2g-2}+(2^{2g-1}-1)(q^{2}-q)^{2g-2} \\
e(\mathcal{M}_{J_{+}}) = & \, e_{2}^{g}/q = (q^{3}-q)^{2g-2}(q^{2}-1)+(2^{2g-1}-1)(q-1)(q^{2}-q)^{2g-2}\\
& \qquad \quad -2^{2g-1}(q+1)(q^{2}+q)^{2g-2}+\frac{1}{2}q^{2g-2}(q-1)\left((q-1)^{2g-1}-(q+1)^{2g-1} \right) \\
e(\mathcal{M}_{J_{-}}) =& \, e_{3}^{g}/q = (q^{3}-q)^{2g-2}(q^{2}-1)+(2^{2g-1}-1)(q-1)(q^{2}-q)^{2g-2}+2^{2g-1}(q+1)(q^{2}+q)^{2g-2} \\
e(\mathcal{M}_{\xi_{\lambda}}) = & \, e_{4,\xi_{\lambda}}^{g}/q-1 = (q^{3}-q)^{2g-2}(q^{2}+q)+(q^{2}-1)^{2g-2}(q+1)+(2^{2g}-2)(q^{2}-q)^{2g-2}q.
\end{align*}
\end{thm}
Note that $e(\mathcal{M}_{-\Id}^{g})$ agrees with the result obtained by arithmetic methods in \cite{mereb:2010}.

\begin{cor} \label{cor:hodgemonoparabolic}
For $g\geq 1$, the behaviour of the E-polynomial of the parabolic character variety $\mathcal{M}_{\xi_{\lambda}}^{g}$ is given by
$$
R(\mathcal{M}_{\xi_{\lambda}})= 
\left( (q^{3}-q)^{2g-2}(q^{2}+q)+(q+1)(q^{2}-1)^{2g-2}-q(q^{2}-q)^{2g-2} \right) T 
+ \left( (2^{2g}-1)q(q^{2}-q)^{2g-2} \right) N.
$$
\end{cor}

\begin{proof}
From Proposition \ref{Polynomialseig} we get that
\begin{align*}
R(\overline{X}{}_{4}^{g})& =(a_{g}+d_{g})T+(b_{g}+c_{g})N \\
& = ((q^{3}-q)^{2g-1}+(q^{2}-1)^{2g-1}-(q^{2}-q)^{2g-1})T + ((2^{2g}-1)(q^{2}-q)^{2g-1})N.
\end{align*}
The result is obtained dividing by $e(\Stab(\xi_{\lambda}))=q-1$.
\end{proof}

To complete the proof of Theorem \ref{thm:main}, it remains the following
\begin{thm} \label{thm:id}
For all $g\geq 1$, we have
\begin{align*}
e(\mathcal{M}_{\Id})=  & \, (q^{3}-q)^{2g-2}+(q^{2}-1)^{2g-2}-q(q^{2}-q)^{2g-2}-2^{2g}q^{2g-2} \\
& +\frac{1}{2}q^{2g-2}(q+2^{2g}-1)((q+1)^{2g-2}+(q-1)^{2g-2}) + \frac{1}{2}q((q+1)^{2g-1}+(q-1)^{2g-1}).
\end{align*}
\end{thm}

\begin{proof}
We need to distinguish between reducible and irreducible orbit since we have to take a GIT quotient to compute $e(\mathcal{M}_{\Id})$ and identify those orbits whose closures intersect. 
We follow the method described in \cite{mamu} and compute the reducible locus. The E-polynomial of the irreducible locus is obtained 
by sustracting the contribution of the reducible part from the E-polynomial of the total space $e_{0}^{g}$. 

A reducible representation given by  $(A_{1},B_{1},A_{2},B_{2},\ldots,A_{g},B_{g}) \in \sldos^{2g}$ is S-equivalent to
\begin{equation} \label{eqn:redrepres}
\left( \begin{pmatrix} \lambda_{1} & 0 \\ 0 & \lambda_{1}^{-1} \end{pmatrix}, \begin{pmatrix} \lambda_{2} & 0 \\ 0 & \lambda_{2}^{-1} \end{pmatrix}, \ldots, \begin{pmatrix} \lambda_{2g} & 0 \\ 0 & \lambda_{2g}^{-1} \end{pmatrix} \right)
\end{equation} 
under the $\ZZ_2$-action $(\lambda_{1},\lambda_{2},\ldots, \lambda_{2g}) \sim (\lambda_{1}^{-1}, \lambda_{2}^{-1}, \ldots, \lambda_{2g}^{-1})$. 
% given by the permutation of the eigenvectors. For each $\lambda_{i}\in \CC^*$, w
We have that $e(\CC^*)^{+}=q, e(\CC^*)^{-}=-1$, so 
\begin{align*}
e(\mathcal{M}_{\Id}^{red}) & = e((\CC^*)^{2g}/\ZZ_2) \\
& = (e(\CC^*)^{+})^{2g}+ \binom{2g}{2}(e(\CC^*)^{+})^{2g-2}(e(\CC^*)^{-})^{2}+\ldots + \binom{2g}{2g-2}(e(\CC^*)^{+})^{2}(e(\CC^*)^{-})^{2g-2}+ (e(\CC^*)^{-})^{2g} \\
& = \frac{1}{2} \left( (q-1)^{2g} +(q+1)^{2g} \right).
\end{align*}

A reducible representation occurs if there is a common eigenvector. With respect to a suitable basis, the representation takes the form
$$
\left( \begin{pmatrix} \lambda_{1} & a_{1} \\ 0 & \lambda_{1}^{-1} \end{pmatrix}, \begin{pmatrix} \lambda_{2} & a_{2} \\ 0 & \lambda_{2}^{-1} \end{pmatrix}, 
\ldots, \begin{pmatrix} \lambda_{2g} & a_{2g} \\ 0 & \lambda_{2g}^{-1} \end{pmatrix} \right),
$$
which is a set parametrized by $(\CC^* \times \CC)^{2g}$. The condition $\prod_{i=1}^{g}[A_{i},B_{i}]=\Id$ is rewritten as
\begin{equation} \label{eqn:redlocusrelation}
\sum_{i=1}^{g} \lambda_{2i}(\lambda_{2i-1}^{2}-1)a_{2i} - \lambda_{2i-1}(\lambda_{2i}^{2}-1)a_{2i-1} = 0.
\end{equation}
There are four cases:
\begin{itemize}
\item $R_{1}$, given by $(a_{1},a_{2},\ldots,a_{2g}) \in \langle (\lambda_{1}-\lambda_{1}^{-1},\lambda_{2}-\lambda_{2}^{-1}, \ldots, \lambda_{2g}-\lambda_{2g}^{-1}) \rangle$, 
and $(\lambda_{1},\lambda_{2}, \ldots, \lambda_{2g}) \neq (\pm 1, \pm 1, \ldots, \pm 1)$. In this case we can conjugate the representation to the diagonal form 
(\ref{eqn:redrepres}) and assume that $a_{i}=0$. 
The stabilizer of this stratum is the set of diagonal matrices $D \subset \pgl$. Writing $A:=(\CC^*)^{2g}- \lbrace (\pm 1, \pm 1, \ldots, \pm 1) \rbrace$, 
the stratum is isomorphic to $(A\times \pgl/D) / \ZZ_2$, where the $\ZZ_2$-action is given by the permutation of the two basis vectors. 
Since $e(\pgl/D)^{+}=q^{2}, e(\pgl/D)^{-}=q$ and
\begin{align*}
e(A)^{+} & = \frac{1}{2}\left((q-1)^{2g}+(q+1)^{2g} \right)-2^{2g} \\
e(A)^{-} & = e(A)-e(A)^{+}= \frac{1}{2}\left( (q-1)^{2g}-(q+1)^{2g} \right),
\end{align*}
we obtain
\begin{align*}
e(R_{1}) & = e(\pgl/D)^{+}e(A)^{+}+e(\pgl/D)^{-}e(A)^{-} \\
& = q^2 \left( \frac{1}{2}((q-1)^{2g}+(q+1)^{2g})-2^{2g}\right) +q\left( \frac{1}{2} ((q-1)^{2g}-(q+1)^{2g}) \right) \\
& = (q^3-q)\frac{1}{2}\left( (q-1)^{2g-1}+(q+1)^{2g-1} \right)-2^{2g}q^2.
\end{align*}

\item $R_{2}$, given by $(a_{1},a_{2},\ldots,a_{2g}) \not \in \langle (\lambda_{1}-\lambda_{1}^{-1},\lambda_{2}-\lambda_{2}^{-1}, \ldots, \lambda_{2g}-\lambda_{2g}^{-1}) \rangle$, 
and $(\lambda_{1},\lambda_{2}, \ldots, \lambda_{2g}) \neq (\pm 1, \pm 1, \ldots, \pm 1)$. Equation (\ref{eqn:redlocusrelation}) defines a hyperplane $H\subset \CC^{2g}$ 
and the condition for $(a_{1},a_{2},\ldots, a_{2g})$ defines a line $l \subset H$. Writing $U'\cong D\times U$ for the upper triangular matrices, 
we have a surjective map $A \times (H - l) \times \pgl \longrightarrow R_{2}$ with fiber isomorphic to $U'$. Hence
\begin{align*}
e(R_{2}) & = ((q-1)^{2g}-2^{2g})(q^{2g-1}-q)(q^{3}-q)/(q^{2}-q) \\
& = (q+1)(q^{2g-1}-q)((q-1)^{2g}-2^{2g}).
\end{align*}

\item $R_{3}$, given by $(a_{1},a_{2},\ldots,a_{2g})=(0,0,\ldots, 0)$, and $(\lambda_{1},\lambda_{2}, \ldots, \lambda_{2g}) = (\pm 1, \pm 1, \ldots, \pm 1)$, corresponding to the case where $A_{i}=B_{i}=\pm \Id$. The stratum consists of $2^{2g}$ points, so
$$
e(R_{3})=2^{2g}.
$$

\item $R_{4}$, given by $(a_{1},a_{2},\ldots,a_{2g}) \neq (0,0, \ldots , 0)$, and $(\lambda_{1},\lambda_{2}, \ldots, \lambda_{2g}) = (\pm 1, \pm 1, \ldots, \pm 1)$. In this case, there is at least a matrix of Jordan type, so the diagonal matrices $D$ act projectivizing the set $(a_{1},a_{2},\ldots,a_{2g})\in \CC^{2g}-\lbrace (0,0,\ldots,0) \rbrace$. The stabilizer is isomorphic to $U$. Therefore
\begin{align*}
e(R_{4}) & = 2^{2g}e(\mathbb{P}^{2g-1})e(\pgl/U) \\
& = 2^{2g}(q^{2g-1}+q^{2g-2}+\ldots + 1)(q^{2}-1) \\
& = 2^{2g}(q^{2g}-1)(q+1).
\end{align*}
\end{itemize}

The total E-polynomial of the reducible locus $R$ is thus
\begin{align*}
e(R) & = e(R_{1})+e(R_{2})+e(R_{3})+e(R_{4}) \\
& =(q^3-q)\left( \frac{1}{2}((q+1)^{2g-1}-(q-1)^{2g-1})+2^{2g}q^{2g-2}+(q-1)(q^{2}-q)^{2g-2} \right).
\end{align*}
We obtain the E-polynomial of the irreducible part as
\begin{align*}
e(I)  = & \, e_{0}^{g}-e(R) \\
 = &\, (q^3-q) \left( (q^3-q)^{2g-2}+(q^{2}-1)^{2g-2}-(q^{2}-q)^{2g-2}+\frac{1}{2}q^{2g-2}(q+2^{2g}-1)((q+1)^{2g-2}+(q-1)^{2g-2}) \right. \\
& \left. -2^{2g}q^{2g-2}-(q-1)(q^{2}-q)^{2g-2}-\frac{1}{2}((q+1)^{2g-1}-(q-1)^{2g-1}) \right),
\end{align*}
and
\begin{align*}
e(\mathcal{M}_{\Id}^{irr}) =&\,  e(I)/(q^{3}-q) \\
= &\,  (q^{3}-q)^{2g-2}+(q^{2}-1)^{2g-2}-q(q^{2}-q)^{2g-2}-2^{2g}q^{2g-2} \\
& +\frac{1}{2}q^{2g-2}(q+2^{2g}-1)((q+1)^{2g-2}+(q-1)^{2g-2}) - \frac{1}{2}((q+1)^{2g-1}-(q-1)^{2g-1}).
\end{align*}
Finally,
\begin{align*}
e(\mathcal{M}_{\Id}) =&\, e(\mathcal{M}_{\Id}^{irr})+e(\mathcal{M}_{\Id}^{red}) \\
=&\,  (q^{3}-q)^{2g-2}+(q^{2}-1)^{2g-2}-q(q^{2}-q)^{2g-2}-2^{2g}q^{2g-2} \\
& +\frac{1}{2}q^{2g-2}(q+2^{2g}-1)((q+1)^{2g-2}+(q-1)^{2g-2}) + \frac{1}{2}q((q+1)^{2g-1}+(q-1)^{2g-1}).
\end{align*}
\end{proof}

%%%%%%%%%%%%%%%%%%%%%%%%%%%%%%%%%%%%%%%%%%%%
\section{Topological consequences} \label{sec:topol}
%%%%%%%%%%%%%%%%%%%%%%%%%%%%%%%%%%%%%%%%%%%%%

In this section, we extract some information from the formulas in Theorem \ref{thm:main}.
We start by a proof of Theorem \ref{thm:balanced}.

\noindent{\textit{Proof of Theorem \ref{thm:balanced}.\/}
In the case of $\cM_{-\Id}, \cM_{J_\pm}$ and $\cM_{\xi_\lambda}$, 
the result follows readily from Proposition \ref{prop:balanced}. For instance, for $\cM_{-\Id}$ we
have that $\cM_{-\Id}=\oX_1/ \PGL(2,\CC)$, where $\oX_1$ and $G=\PGL(2, \CC)$ are of balanced type.
Hence the classifying space $BG$ is also of balanced type and there is a homotopy fibration
$\oX_1 \to \oX_1/G \to BG$. The Leray spectral sequence gives that $\oX_1/G$ must be of balanced
type (a similar argument appears in the proof of Proposition 7.2 of \cite{Mu-e}).

In the case of $\cM_{\Id}$, the description in Theorem \ref{thm:id} yields that 
$R$ is of balanced type. Hence $I$ is also of balanced type. The same argument as above 
proves that $\cM_{\Id}^{irr}=I/\PGL(2,\CC)$ is of balanced type. As $\cM_{\Id}^{red}$ is
clearly of balanced type, so is $\cM_{\Id}$.
\hfill $\Box$

\medskip

 \begin{cor}
 Let $X$ be a complex curve of genus $g\geq 2$. The Euler characteristic of $\cM_C=\cM_C(\SL(2,\CC))$ is given by
\begin{align*}
&\chi(\mathcal{M}_{\Id})=  2^{4g-3}-3\cdot 2^{2g-2} \\
&\chi(\mathcal{M}_{-\Id}) =  -2^{4g-3} \\
&\chi(\mathcal{M}_{J_{+}}) = -2^{4g-2} \\
&\chi(\mathcal{M}_{J_{-}}) =2^ {4g-2}\\
&\chi(\mathcal{M}_{\xi_{\lambda}}) =  0.
\end{align*}
\end{cor}

\begin{proof}
The Euler characteristic is obtained by setting $q=1$ in $e(\cM_C)$ given in  Theorem \ref{thm:main}. 
\end{proof}

 \begin{cor}
 Let $X$ be a complex curve of genus $g\geq 2$. Then $\cM_{\Id}$ and $\cM_{-\Id}$ are of dimension $6g-6$ and
$\cM_{J_+}$, $\cM_{J_-}$ and $\cM_{\xi_\lambda}$ are of dimension $6g-4$. All of them have a unique component of
maximal dimension.
\end{cor}

\begin{proof}
From Theorem \ref{thm:main}, we get
\begin{align*}
&e(\mathcal{M}_{\Id})=  q^{6g-6}+ \ldots + 1 \\
&e(\mathcal{M}_{-\Id}) =  q^{6g-6}+ \ldots + 1\\
&e(\mathcal{M}_{J_{+}}) =   q^{6g-4}+ \ldots +(1-2^{2g-1})q^{2g-2} \\
&e(\mathcal{M}_{J_{-}}) = q^{6g-4}+ \ldots +(2^{2g}-1) q^{2g-1}\\
&e(\mathcal{M}_{\xi_{\lambda}}) =   q^{6g-4}+ \ldots + 1
\end{align*}
where we have written the monomials of maximum and minimum degrees in each case.
The degree of the polynomial gives the dimension of the character variety, and the coefficient (which is always $1$)
gives the number of irreducible components.
\end{proof}

 \begin{cor}
 Let $X$ be a complex curve of genus $g\geq 1$. Then $e(\mathcal{M}_{-\Id})$, $e(\mathcal{M}_{\xi_{\lambda}})$, 
and its invariant and non-invariant part given in Cororally \ref{cor:hodgemonoparabolic}, are palindromic polynomials.
\end{cor}

\begin{proof}
Let $d=6g-4$. If we write $R(\mathcal{M}_{\xi_{\lambda}})=AT+BN$, with
\begin{align*}
A & = (q^3-q)^{2g-2}(q^2+q)+(q+1)(q^2-1)^{2g-2}-q(q^{2}-q)^{2g-2}, \\
B & = (2^{2g}-1)q(q^{2}-q)^{2g-2},
\end{align*}
given in Corollary \ref{cor:hodgemonoparabolic}, then one only has to check that $q^dA(q^{-1})= A(q)$ and $q^d B(q^{-1})=B(q)$, which is straightforward. 

The computation for $e(\mathcal{M}_{-\Id})$ is analogous and it is also given in \cite[Section 4.4]{mereb:2010}.
\end{proof}

\end{document}